\documentclass[a4paper, reqno]{amsart}
\usepackage[utf8]{inputenc}
\usepackage[T2A]{fontenc}
\usepackage{amsmath,amssymb,amsthm}
\usepackage[a4paper,hmargin=2.5cm,vmargin=2.5cm]{geometry}
\usepackage[russian, english]{babel}
\usepackage{mathtext}
\usepackage[unicode=true]{hyperref}
\usepackage{amscd}
\usepackage{scalerel}
\usepackage[all]{xy}
\usepackage{yfonts}
\usepackage{enumitem}
\usepackage{extarrows}
\usepackage{amsfonts}
\usepackage{graphicx}
\usepackage{mathrsfs}
\usepackage{amsmath}
\usepackage{amssymb}
\usepackage{pstricks}
\usepackage{verbatim}
\usepackage{mathtools}

\def\colon{\mathpunct{:}}

\title{A lower bound in the problem of realization of cycles}

\author{Vasilii Rozhdestvenskii}
\thanks{This work was performed at the Steklov International Mathematical Center and supported by the Ministry of Science and Higher Education of the Russian Federation (agreement no.\ 075-15-2022-265).}
\address{{\sloppy
\parbox{0.9\textwidth}{
Steklov Mathematical Institute of Russian Academy of Sciences, Moscow, Russia
\\[5pt]
Skolkovo Institute of Science and Technology, Moscow, Russia
\\[5pt]
National Research University Higher School of Economics, Moscow, Russia
}\smallskip}}
\email{vrozhd@mi-ras.ru  \medskip}

\begin{document}

\newtheorem{Th}{Theorem}
	\newtheorem*{Th*}{Theorem} 
	\newtheorem{Lem}{Lemma}
	\newtheorem*{Lem*}{Lemma}
	\newtheorem{Cor}{Corollary}
	\newtheorem*{Cor*}{Corollary}
	\newtheorem*{Frame}{Frame}
	\newtheorem*{Denom}{Denomination}
	\newtheorem*{Prop*}{Proposition}
	\newtheorem{Prop}{Proposition}
	\theoremstyle{definition}
	\newtheorem{definition}{Definition}
	\newtheorem*{definition*}{Definition}
	\newtheorem{Rem}{Remark}
	\newtheorem{Conv}{Convention}

	\makeatletter
\newcommand{\colim@}[2]{%
  \vtop{\m@th\ialign{##\cr
    \hfil$#1\operator@font colim$\hfil\cr
    \noalign{\nointerlineskip\kern1.5\ex@}#2\cr
    \noalign{\nointerlineskip\kern-\ex@}\cr}}%
}
\newcommand{\colim}{%
  \mathop{\mathpalette\colim@{\rightarrowfill@\textstyle}}\nmlimits@
}
\makeatother
	
	\def\rk{\mathop{\mathrm{rk}}}
	\def\id{\mathord{\mathrm{id}}}
	\def\pt{\mathord{\mathrm{pt}}}
	\def\Ker{\mathop{\mathrm{Ker}}}
	\def\Im{\mathop{\mathrm{Im}}}
	\def\Tor{\mathop{\mathrm{Tor}}}
	\def\Ext{\mathop{\mathrm{Ext}}}
	\def\Tors{\mathop{\mathrm{Tors}}}
	\def\ord{\mathop{\mathrm{ord}}}
	\def\ex{\mathop{\mathrm{ex}}}
	\def\Indet{\mathop{\mathrm{Indet}}}
	\def\Hom{\mathop{\mathrm{Hom}}}
	\def\Ann{\mathop{\mathrm{Ann}}}
	\def\MU{\mathord{\mathrm{MU}}}
	\def\bu{\mathord{\mathrm{bu}}}
	\def\MSO{\mathord{\mathrm{MSO}}}
	\def\SO{\mathord{\mathrm{SO}}}
	\def\Sq{\mathord{\mathrm{Sq}}}
	\def\U{\mathord{\mathrm{U}}}
	\def\MG{\mathord{\mathrm{MG}}}
	\def\ch{\mathord{\mathrm{ch}}}
	\def\BU{\mathord{\mathrm{BU}}}
	\def\St{\mathord{\mathrm{St}}}
	\def\K{\mathord{\mathrm{K}}}
	\newcommand{\Z}{\mathbb{Z}}
	\newcommand{\Q}{\mathbb{Q}}
	\let\le\leqslant
	\let\ge\geqslant

\begin{abstract}
We consider the classical Steenrod problem on realization of integral homology classes by continuous images of smooth oriented manifolds. Let $k(n)$ be the smallest positive integer such that any integral $n$-dimensional homology class becomes realizable in the sense of Steenrod after multiplication by~$k(n)$. The best known upper bound for $k(n)$ was obtained independently by G.~Brumfiel and V.~Buchstaber in 1969. All known lower bounds for $k(n)$ were very far from this upper bound. The main result of this paper is a new lower bound for $k(n)$ which is asymptotically equivalent to the Brumfiel---Buchstaber upper bound (in the logarithmic scale). For $n<24$ we prove that our lower bound is exact. Also we obtain analogous results for the case of realization of integral homology classes by continuous images of smooth stably complex manifolds.
\end{abstract}
\maketitle

  \section{Introduction}
One of the classical problems in topology is the Steenrod problem on realization of cycles, formulated at the end of the $1940$s (see  \cite{Steenrod}): 
\[
\parbox[c]{14cm}{
For a given integral homology class $x\in H_n(X;\Z)$ of a topological space $X$, do there exist a smooth closed oriented manifold $M$ and a continuous map  $f\colon M^n \to X$ such that $f_*[M^n]=x$?}
\] 

This question is closely related to the question of realization of integral homology classes of a smooth manifold by smooth oriented submanifolds. Indeed, the space $X$ (which, without loss of generality, can be considered to be a finite CW complex) can be embedded in a high-dimensional Euclidean space $\mathbb{R}^N$ and one can take its regular neighbourhood $V(X)$. Then the problem of realization of  homology classes in $X$ in the sense of Steenrod reduces to the problem of realization of homology classes in the manifold $V(X)$ by submanifolds.
Moreover, taking the double $V(X)\cup_{\partial V(X)}V(X)$ of the manifold $V(X)$, one can reduce the Steenrod problem for the space $X$ to the problem of realization of homology classes by submanifolds in a closed oriented manifold. 

In modern terminology, the Steenrod problem is the problem of describing the image of the natural homomorphism $\mu_{\SO*}\colon \MSO_n(X)\to H_n(X;\Z)$. It is natural to pose a similar problem for other bordism theories, among which complex bordism is of particular interest. 

First results on the Steenrod problem were obtained by R. Thom in 1954 (see \cite{Thom}). Namely, he constructed the first examples of Steenrod-non-realizable classes and proved that an arbitrary $n$-dimensional integral homology class can be realized after multiplication by some natural number, which can be chosen uniformly for all $n$-dimensional classes. In connection with this result, the following problem arose:
\[
\parbox[c]{14cm}{
For each dimension $n$ find the smallest positive integer $k_{\SO}(n)$ such that every integral \linebreak $n$-dimensional homology class can be realized in the sense of Steenrod after multiplication by~$k_{\SO}(n)$. 
}
\]
And there is a similar problem for complex bordism: 
\[
\parbox[c]{14cm}{
For each dimension $n$ find the smallest positive integer $k_{\U}(n)$ such that every integral \linebreak  $n$-dimensional homology class can be realized by a continuous image of a stably complex manifold after multiplication by $k_{\U}(n)$. 
}
\]

 The main result of the paper is a significant improvement of previously known lower bounds for the numbers $k_{\SO}(n)$ and $k_{\U}(n)$, as well as a slight improvement of upper bounds. As a consequence,  we obtain such bounds for the numbers $k_{\SO}(n)$ and $k_{\U}(n)$ which are asymptotically exact in the logarithmic scale. Namely, the following theorem is proved.
\begin{Th}[Theorem~\ref{Main theorem intro}]\label{main theorem beginning introduction}
	For every dimension $n$ the following inequalities hold:
	\[
	\left[\frac{n-1}{2}\right]\hskip-1mm \mathord{\vcenter{\hbox{\LARGE\textup !}}}  \le k_{\U}(n) \le \prod_{p\ \text{prime}}p^{\left[\frac{n-3}{2(p-1)}\right]}
	\]
	and
	\[
	\prod_{p\ \text{prime},\ p>2} p^{ \sum_{i=1}^{\infty} \left[\frac{n-1}{2p^i}\right]} \le k_{\SO}(n) \le \prod_{p\ \text{prime},\ p>2}p^{\left[\frac{n-3}{2(p-1)}\right]}.
	\]
\end{Th}
Note that if $x$ is a homology class and classes $mx$ and $m'x$ are realizable, then the class $\gcd(m,m')x$ is also realizable. Throughout the paper, by saying that a number $k$ is a lower (resp. upper) bound for the numbers $k_{\SO}(n)$ or $k_{\U}(n)$, we always mean that there exists a homology class which cannot be realized after multiplication by a positive integer less than $k$ (resp. any class can be realized after multiplication by $k$). So in the above theorem all inequalities actually mean divisibilities of one number by another. 

Let us now dwell in more detail on previously known results. In \cite{Nov}, S.\,P. Novikov showed that the number $k_{\SO}(n)$ is odd and all obstructions to the realization of a class $x\in H_m(X;\Z)$ in the sense of Steenrod lie in the $p$-torsion of the homology groups $H_ {m-2r(p-1)-1}(X;\Z)$ where $p$ is an odd prime and~$r>0$. Similarly, all obstructions to the realization of a class $x\in H_m(X;\Z)$ by a stably complex manifold lie in the $p$-torsion of the homology groups $H_{m-2r(p-1)-1}(X;\Z)$ where $p$ is a prime and $r>0$. In particular, if the mentioned homology groups have no $p$-torsion, then all obstructions to the realization vanish and the class $x$ is realizable. In \cite{BP}, E. Brown and F. Peterson proved that for $p >2$ the obstructions in the prime~$p$ to the realization of homology classes by oriented and stably complex manifolds coincide. Thus the number $k_{\SO}(n)$ can be explicitly expressed in terms of the number $k_{\U}(n)$, namely $k_{\SO}(n)$ is the maximal odd divisor of $k_{\U}(n)$. Probably, this result can also be extracted from \cite{Nov} but it was not explicitly stated~there.

After Novikov's result the following question arose. Let $n$ be a positive integer. Suppose that the homology groups $H_{m-2r(p-1)-1}(X;\Z)$ of a space $X$ do not contain $p$-torsion for any prime~$p$ and any~$r$ satisfying $2r(p-1)+1>n$. What is the minimal positive integer $A$ such that for any homology class $x\in H_m(X;\Z)$ the class $Ax$ is realizable? This problem was solved by V.\,M. Buchstaber. Namely, he proved the following theorem. 

\begin{Th*}[Buchstaber \cite{Buch.2}]
Let $X$ be a space such that the groups $H_{m-2r(p-1)-1}(X;\Z)$ have no $p$-torsion for any prime $p$ and any $r$ with $2r(p-1)+1>n$. Then any integral homology class $x\in H_m(X;\Z)$ can be realized by a continuous image of a stably complex manifold after multiplication by 
$$
k^s_{\U}(n)=\prod_{p\ \text{prime}}p^{\left[\frac{n-1}{2(p-1)}\right]}.
$$ 
This bound is exact in $n$: for each $n$ there exist a number $m$, a space $X$ that satisfies the above condition and a homology class $x\in H_m(X;\Z)$ such that $x$ cannot be realized by a stably complex manifold after multiplication by a positive integer less than $k^s_{\U}(n)$. 

In particular, for $n=m-1$, any $m$-dimensional homology class can be realized by a continuous image of a stably complex manifold after multiplication by $k^s_{\U}(m-1)$. Thus the number~$k_{\U}(m)$ divides the number~$k^s_{\U}(m-1)$.
\end{Th*}
Using the above result of Novikov, it is easy to obtain a similar theorem for the case of realization by oriented manifolds. 

Independently, G. Brumfiel obtained a similar result, but under a weaker assumption.
\begin{Th*}[Brumfiel \cite{Brumfiel}]
Any $(q+n)$-dimensional homology class $x\in H_{q+n}(X;\Z)$ in a $(q-1)$-connected space $X$ can be realized in the sense of Steenrod after multiplication by
$$
k^s_{\SO}(n)=\prod_{p\ \text{prime},\ p>2}p^{\left[\frac{n-1}{2(p-1)}\right]}.
$$
This bound is exact in $n$: for each $n$ there exist a $(q-1)$-connected space $X$ and a homology class $x\in H_{q+n}(X;\Z)$ such that $x$ cannot be realized in the sense of Steenrod after multiplication by a positive integer less than $k^s_{\SO}(n)$.

In particular, for $q=1$, any $m$-dimensional homology class can be realized in the sense of Steenrod after multiplication by $k^s_{\SO}(m-1)$. Thus the number  $k_{\SO}(m)$ divides the number $k^s_{\SO}(m-1)$.
\end{Th*}

For an integer $a$ and a prime $p$ we denote by $\nu_p(a)$ the degree of $p$ in the prime factorization of $a$. The results of Buchstaber and Brumfiel provide the following inequality: 
\begin{equation}\label{BB upper}
 \nu_p\bigl(k_{\U}(n)\bigr) \le \left[\frac{n-2}{2(p-1)}\right].
\end{equation}
Note that this bound is not exact. Indeed, in \cite{Thom} Thom proved that $k_{\SO}(6)=1$ and so $\nu_3\bigl(k_{\U}(6)\bigr)=0$, while the right-hand side of \eqref{BB upper} is equal to $1$ in this case. This was noticed by Brumfiel in \cite{Brumfiel} and he raised the question about the exact computation of $k_{\SO}(n)$. 

Since the obstructions to the realization are stable in degrees $n-r$, provided that $r<n/2$, the results of Buchstaber and Brumfiel also imply the following inequality: 
\begin{equation}\label{BB low}
\frac{1}{2}\left[\frac{n-1}{2(p-1)}\right] \le \nu_p\bigl(k_{\U}(n)\bigr). 
\end{equation} 
Note that this bound was the best known until now. 

Let us now return to results of the present paper. Theorem~\ref{main theorem beginning introduction} can be rewritten in the following equivalent form. 
\begin{Th}[Theorem~\ref{win} and Theorem~\ref{upper boundry for realization}]\label{Main theorem intro} 
	For every dimension $n$ and every prime $p$ the following inequalities hold:
	\[
	\sum_{i=1}^{\infty} \left[\frac{n-1}{2p^i}\right] \le  \nu_p\bigl(k_{\U}(n)\bigr) \le \left[\frac{n-3}{2(p-1)}\right].
	\]
\end{Th}
The upper bound here is a slight improvement of the Brumfiel---Buchstaber bound \eqref{BB upper}. But  the following corollaries demonstrate that our lower bound is much better than the previously known bound~\eqref{BB low}.
\begin{Cor}
If there exists a positive integer $l$ such that $2p^l+1\le n<2p^l+1+2(p-1)$, then the lower and upper bounds in Theorem~\ref{Main theorem intro} coincide and therefore 
\[
	\nu_p\bigl(k_{\U}(n)\bigr)=\frac{p^l-1}{p-1}.
\]
\end{Cor}

Thus, for every prime $p$ there is an infinite series of dimensions $n$ such that our lower bound for the number $\nu_p\bigl(k_{\U}(n)\bigr)$ is exact. Unfortunately, these series of dimensions are different for different primes, and therefore it is not clear whether there exists an infinite series of dimensions for which the lower bound is exact for the entire number $k_{\U}(n)$.
\begin{Cor}
The ratio of the upper bound to the lower bound for the number  $\nu_p\bigl(k_{\U}(n)\bigr)$ from Theorem~\ref{Main theorem intro} tends to~$1$ as~$n$ tends to infinity.
\end{Cor}
\begin{proof}
It is readily verified that for a positive integer $n$ and a prime $p$ the following holds. 
\begin{multline*}
\left[\frac{n-3}{2(p-1)}\right] - \sum_{i=1}^{\infty} \left[\frac{n-1}{2p^i}\right] \le
\frac{n-1}{2(p-1)}  -\frac{1}{p-1} - \sum_{i=1}^{\infty} \left[\frac{n-1}{2p^i}\right] = \\ 
=\sum_{i=1}^{\infty}\left(\frac{n-1}{2p^i} - \left[\frac{n-1}{2p^i}\right]\right)-\frac{1}{p-1}= 
\sum_{i=1}^{\left[\log_p((n-1)/2)\right]}\left(\frac{n-1}{2p^i} - \left[\frac{n-1}{2p^i}\right]\right)+\\ 
+\sum_{i=\left[\log_p((n-1)/2)\right]+1}^{\infty}\frac{n-1}{2p^i} \ \ - \ \frac{1}{p-1} < \log_p\left(\frac{n-1}{2}\right) +1
\end{multline*}
Therefore 
\[
1\le \frac{\left[\frac{n-3}{2(p-1)}\right]}{\sum_{i=1}^{\infty} \left[\frac{n-1}{2p^i}\right]}\le 1-\frac{\log_p\left(\frac{n-1}{2}\right) +1}{\sum_{i=1}^{\infty} \left[\frac{n-1}{2p^i}\right]}, 
\]
so this ratio tends to $1$ as n tends to infinity. 
\end{proof}

One can also consider the problem of realization of homology classes by continuous images of manifolds from a certain class $\mathcal{X}$. For instance, one can take for $\mathcal{X}$ the class of stably complex manifolds admitting certain additional structure in stable tangent bundle. 

\begin{Cor}\label{537}
Let $\mathcal{X}$ be a class of smooth closed oriented manifolds such that any stably parallelizable manifold lies in $\mathcal{X}$ and any manifold $M\in  \mathcal{X}$ can be endowed with stably complex structure. Denote by~$k_{\mathcal{X}}(n)$ the minimal positive integer such that every integral $n$-dimensional homology class can be realized by a continuous image of a manifold from $\mathcal{X}$ after multiplication by $k_{\mathcal{X}}(n)$. Then the number $k_{\mathcal{X}}(n)$ is divisible by $\left[ \frac{n-1}{2}\right]!$ and the ratio 
\begin{equation}\label{9999}
\frac{\nu_p\bigl(k_{\mathcal{X}}(n)\bigr)}{\sum_{i=1}^{\infty} \left[\frac{n-1}{2p^i}\right]}
\end{equation}
tends to $1$ as $n$ tends to infinity. 
\end{Cor}
\begin{proof}
Denote by $k_{e}(n)$ the minimal positive integer such that every integral $n$-dimensional homology class can be realized by a continuous image of a stably parallelizable manifold after multiplication by~$k_{e}(n)$. Since any manifold from $\mathcal{X}$ can be endowed with stably complex structure, then the number $k_{\mathcal{X}}(n)$ is divisible by $k_{\U}(n)$. Thus Theorem~\ref{main theorem beginning introduction} implies that $k_{\mathcal{X}}(n)$ is divisible by $\left[ \frac{n-1}{2}\right]!$. Since any stably parallelizable manifold lies in $\mathcal{X}$, then the number $k_{\mathcal{X}}(n)$ divides the number $k_{e}(n)$. By Theorem $1.3$ of~\cite{Akhill Mathew}, for every prime $p$ and every dimension  $n$ the following inequality holds
$$\nu_p\bigl(k_{e}(n)\bigr)\le \left[ \frac{n+1}{2(p-1)}\right]+3.$$
 Thus we obtain the inequality $\nu_p\bigl(k_{\mathcal{X}}(n)\bigr)\le \left[ \frac{n+1}{2(p-1)}\right]+3$. It is readily verified that 
\[
\left[ \frac{n+1}{2(p-1)}\right]+3 - \sum_{i=1}^{\infty} \left[\frac{n-1}{2p^i}\right] \le \left[\frac{n-3}{2(p-1)}\right] + 5 - \sum_{i=1}^{\infty} \left[\frac{n-1}{2p^i}\right] < \log_p\left(\frac{n-1}{2}\right) +6.
\]
Thus the ratio \eqref{9999} tends to $1$ as $n$ tends to infinity. 
\end{proof}

A second result of the paper is an explicit computation of $k_{\SO}(n)$ and $k_{\U}(n)$ for small dimensions.
\begin{Th}[Corollary \ref{low dim}]\label{small n intro}
	If $n<12$, then
	$$k_{\U}(n)=\left[\frac{n-1}{2}\right]\hskip-1mm \mathord{\vcenter{\hbox{\LARGE\textup !}}} .$$
	If $n<24$, then the number $k_{\SO}(n)$ is equal to the maximal odd divisor of $\left[\frac{n-1}{2}\right]!$, that is
	$$k_{\SO}(n)=\prod_{p\ \text{prime},\ p>2} p^{ \sum_{i=1}^{\infty} \left[\frac{n-1}{2p^i}\right]}.$$
\end{Th}
The values of $k_{\SO}(n)$ for $n<24$ are given in the following table. 
\[
\begin{tabular}{|c|c|c|c|c|c|c|c|c|c|c|c|c|}
\hline 
$n$ & 1 & 2 & 3 & 4 & 5 & 6 & 7 & 8 & 9 & 10 & 11 & 12  \\ 
\hline 
$k_{\SO}(n)$ & 1 & 1 &1 &1 &1 &1 &3 &3  &3 &3 &15 &15  \\
\hline
\hline
$n$ & 13 & 14 & 15 & 16 & 17 & 18 & 19 & 20 & 21 & 22 & 23 & 24 \\ 
\hline 
$k_{\SO}(n)$ & 45 & 45 &315 &315 &315 &315 &2835 &2835  &14175 &14175 &155925 & ?\\
\hline
\end{tabular}
\]
If $n=24$, then the number $k_{\SO}(24)$ is equal to either $155925$ or $3\cdot155925$. 

Most of the work is devoted to proving the lower bound for the number $\nu_p\bigl(k_{\U}(n)\bigr)$ from Theorem~\ref{Main theorem intro}. The proof consists in an explicit construction of a homology class which can be realized only after multiplication by the desired number. The underlying space of the homology class looks as follows. 
\begin{itemize}
	\item Denote by $\bu(2q,s)$ the classifying space $\BU$ of the infinite unitary group with killed all the homotopy groups $\pi_n$ with $n<2q-1$ or $n>2q+2s-1$. 
	 Denote by $\bu(2q-1,s)$ the loop space $\Omega \bu(2q,s)$. Denote by $\bu_{(p)}(q,s)$ the space $\bu(q,s)$ localized at prime $p$.
	\item In section $4$ we construct a map
	   \[
	   \pi \colon \bu_{(p)}(q,s)\to \prod_{i=0}^{s-1} K(\Q, q+2i)
	   \]
	   such that $\pi$  induces an isomorphism of rational homology groups. Denote by $R_p(q,s)$ the homotopy fibre of $\pi$. 
\end{itemize}
Then the lower bound in Theorem~\ref{Main theorem intro} is attained at certain homology class of the space 
  $$R_p\Biggr(n-2\Biggr(\sum_{i=1}^{\infty} \left[\frac{n-1}{2p^i}\right]\Biggl)(p-1), \  \Biggl(\sum_{i=1}^{\infty} \left[\frac{n-1}{2p^i}\right]\Biggr)(p-1)\Biggl).$$

 The present paper is organized as follows. In subsection $2.1$, for every prime $p$ we construct obstructions~$\Phi_r$ to the realization and prove their main properties. These obstructions are nothing else but differentials in the Atiyah---Hirzebruch spectral sequence (AHSS) for complex $\K$-theory localized at prime~$p$. Nevertheless, for our purposes it is more convenient to consider $\Phi_r$ as integral higher cohomological operations given by Postnikov invariants of the spectrum $\bu_{(p)}$. Here $\bu_{(p)}$ is a spectrum that represents the connected complex $\K$-theory localized at prime~$p$.  Considering integral obstructions rather than $\bmod \ p$ obstructions is crucial for the present paper. Since the homology functor is not representable, it is quite difficult to study the action of higher cohomological operations on homology. Hence we use the Kronecker pairing between homology and cohomology in order to reduce questions on homology to certain questions on cohomology. Following this strategy, in subsection $2.1$ we introduce higher operations~$\widetilde{\Phi}_r$ such that $\beta\widetilde{\Phi}_r=\Phi_r$, where $\beta$ is the Bockstein homomorphism for the short exact sequence 
\[
0 \to \Z_{(p)} \to \Q \to \Q/\Z_{(p)} \to 0
\]
of coefficient groups. In subsection $2.2$ we prove the duality theorem (Theorem~\ref{duality}) which states, roughly speaking, that if $\Phi_r$ is defined on classes $x\in H^n(X;\Z_{(p)})$ and $y\in H_{n+2r(p-1)}(X;\Z_{(p)})$, then we have the equality $\bigl\langle y, \widetilde{\Phi}_r(x)\bigr\rangle  =-\bigl\langle \widetilde{\Phi}_r(y), x \bigr\rangle$ of elements of $\Q/\Z_{(p)}$. Finally, in subsection $2.3$ we prove the following theorem, which plays a key role in the proof of the lower bound in Theorem~\ref{Main theorem intro}. 
\begin{Th}[Theorem~\ref{Lem for example}]\label{importan for example}
Let $X$ be a CW complex satisfying the following conditions. 
\begin{enumerate}
\item The rational reduced homology groups $\widetilde{H}_*(X;\Q)$ vanish. 
\item There exists a class $x\in H^{n}\bigl(X;\Z_{(p)}\bigr)$ such that the operation $\Phi_r$ is defined on $x$ and $\tilde{\beta}\chi(P^r)(x)\ne 0$, where $P^r$ is the Steenrod cyclic reduced power, $\chi$ is the antipode of the Steenrod algebra~$\mathcal{A}_{p}$, and $\tilde{\beta}\colon H^*(X;\Z/p\Z)\to H^{*+1}(X;\Z_{(p)})$ is the Bockstein homomorphism.
\end{enumerate}
Then for any homology class $a\in H_{n+2r(p-1)}(X;\Z)$ satisfying $\langle a,\chi(P^r)(x)\rangle  \ne0$ and for any integer $m$ coprime with $p$, the homology class $mp^{r-1}a$ cannot be realized by a continuous image of a stably complex manifold. Moreover, such a class $a$ always exists.
\end{Th}

In section $3$ we study the action of the operations $\chi(P^r)$ on cohomology classes. In section $4$, using results of section $3$ and Theorem~\ref{importan for example}, we construct homology classes on which the lower bound in Theorem~\ref{Main theorem intro} is achieved. 

In section $5$ we prove the upper bound in Theorem~\ref{Main theorem intro} and Theorem~\ref{small n intro} by using the Atiyah---Hirzebruch spectral sequence for complex bordism. 
 
 I am grateful to my advisor Alexander Gaifullin for setting the problem and for valuable discussions. 
 
 \section{Partial operations \texorpdfstring{$\Phi_r$}{Pr} and \texorpdfstring{$\widetilde{\Phi}_r$}{Pr}}
 In subsections $2.1$ and $2.2$ we work in the stable homotopy category. All spectra are assumed to be CW spectra. Let $p$ be a fixed prime number. Denote by $\Z_{(p)}$ the ring of integers localized at the prime ideal $(p)$. 
 \subsection{Construction of operations \texorpdfstring{$\Phi_r$}{Pr} and \texorpdfstring{$\widetilde{\Phi}_r$}{Pr}}
 
 Let us fix some notation.
 \begin{itemize}
	\item Denote by $\bu_{(p)}$ a spectrum which represents the connective complex $\K$-theory localized at $p$. Denote by $u\in H^0\bigl(\bu_{(p)};\Z_{(p)}\bigr)=\Z_{(p)}$ the standard generator.
	\item Let $\bu_{(p)}(s)$ be the $(2s-2)^{\text{th}}$ stage of the Postnikov tower of $\bu_{(p)}$. Then 
	\[\pi_i\bigl(\bu_{(p)}(s)\bigr)=\begin{cases}
	\Z_{(p)},  & i=2a, \  0\le a<s; \\
	0, & \text{otherwise}.
	\end{cases}
	\]
	\item Denote by $u_{s} \in H^{0}\bigl(\bu_{(p)}(s);\Z_{(p)}\bigr)=\Z_{(p)}$ a generator that goes to $u$ under the canonical map $$\bu_{(p)}\to \bu_{(p)}(s).$$ 
	\item Denote by $k^{2s+1}\in H^{2s+1}\bigl(\bu_{(p)}(s);\Z_{(p)}\bigr)$ the Postnikov invariants of $\bu_{(p)}$. 
	\item Consider the following commutative diagram of coefficient groups. 
	\[
	\xymatrix{ 
	0 \ar[r]  & \Z_{(p)} \ar[r]^-{\times p} \ar[d]^{\id} & \Z_{(p)} \ar[r]^-{\bmod p} \ar[d]^-{1 \mapsto 1/p} & \Z/p\Z \ar[r]  \ar[d]_{\rho}^-{1 \mapsto 1/p} & 0 
	\\
	0 \ar[r] & \Z_{(p)} \ar[r]_{1\mapsto 1} & \Q \ar[r]_-{\bmod \Z_{(p)} } & \Q/\Z_{(p)} \ar[r] & 0
	}
	\]
	Note that both rows are short exact sequences. Denote by $\tilde{\beta}$ the Bockstein homomorphism for the first row, and by $\beta$ for the second one. Denote by $\rho_*$ the map on (co)homology groups induced by the homomorphism $\rho \colon \Z/p\Z \to \mathbb{Q}/\Z_{(p)}, \ \rho(1)=1/p$ of the coefficient groups. Thus we have the following commutative diagram for cohomology groups
	\begin{gather}
	\begin{aligned}
	\xymatrix{
	H^*(X;\Z_{(p)}) \ar[r] \ar[d] &  H^*(X;\Z/p\Z) \ar[r]^-{\tilde{\beta}} \ar[d]^-{\rho_*} &  H^{*+1}(X;\Z_{(p)}) \ar[d]^-{\id}  \\
	H^*(X;\Q) \ar[r] &  H^*(X;\Q/\Z_{(p)}) \ar[r]^-{\beta}  &  H^{*+1}(X;\Z_{(p)}) 
	}
	\end{aligned}
	\label{comm diagram for coeff}
	\end{gather}
        and a similar one for homology groups.  In particular, $\tilde{\beta}=\beta\rho_*$.
	\item Let us choose and fix throughout the paper a class $\widetilde{k}^{2s+1}\in H^{2s}\bigl(\bu_{(p)}(s);\mathbb{Q}/\Z_{(p)}\bigr)$ such that $\beta\widetilde{k}^{2s+1}=k^{2s+1}.$ Such а class exists since the Postnikov invariants of a spectrum have finite orders. 
\end{itemize}
 
For a finite spectrum $X$ let us denote by $D(X)$ the Spanier---Whitehead dual of $X$. The following properties of the duality are well known.
\begin{itemize}
	\item For any (co)homology theory $h$ there is a natural isomorphism 
	\[
	D\colon h_n\bigl(X\bigr) \to h^{-n}\bigl(D(X)\bigr).
	\]
	\item There is a map $\eta\colon \mathbb{S} \to D(X)\wedge X$ such that for any classes $x\in H_n\bigl(X; G \bigr)$ and $y\in H^n\bigl(X; H \bigr)$ we have that
	\[
	\langle x,y\rangle  =\eta^*\bigl(D(x)\times y\bigr)\in H^0\bigl(\mathbb{S}; G \otimes_{\Z} H \bigr)=G \otimes_{\Z} H ,\]
	where $\mathbb{S}$ denotes the sphere spectrum. 
\end{itemize}
 The cohomology classes $k^{2s+1}$ and $\widetilde{k}^{2s+1}$ induce higher cohomological operations in the following way. 
 
 \begin{definition}\label{def. def of integer higher operations} 
	Let $X$ be a spectrum. Define the action of the classes $k^{2s+1}$ and  $\widetilde{k}^{2s+1}$ on a cohomology class  $x\in H^n\bigl(X;\Z_{(p)}\bigr)$ by the following formulas: 
	\[
	k^{2s+1}(x)=\{f^*\bigl(k^{2s+1}\times \sigma^n \bigr)\}\subset H^{n+2s+1}\bigl(X;\Z_{(p)}\bigr),
	\]
	\[
	\widetilde{k}^{2s+1}(x)=\{f^*\bigl(\widetilde{k}^{2s+1}\times \sigma^n \bigr)\}\subset H^{n+2s}\bigl(X;\mathbb{Q}/\Z_{(p)}\bigr),
	\]
	where $\sigma\in H^1(\mathbb{S}^1;\Z_{(p)})$ is the generator and $f\colon X \to \bu_{(p)}(s)\wedge \mathbb{S}^n$ runs over all maps such that 
	$$f^*\bigl(u_s\times \sigma^n \bigr)=x.$$
	
	Since the stable homotopy category is additive, then the set $k^{2s+1}(x)$ is either empty or a coset of~$k^{2s+1}(0)$, where $0$ is considered as an element of $H^n\bigl(X;\Z_{(p)}\bigr)$. We denote the subgroup $k^{2s+1}(0)$ by $\Indet_X^n(k^{2s+1})$ and call it the \textit{indeterminacy} of the operation $k^{2s+1}$. Since the underlying space $X$ will always be clear from the context we will omit the corresponding index. By analogy, the set $\widetilde{k}^{2s+1}(x)$ is either empty or a coset of $\widetilde{k}^{2s+1}(0)=\Indet^n\bigl(\widetilde{k}^{2s+1}\bigr)$. 
	
	 Define the action of the classes $k^{2s+1}$ and  $\widetilde{k}^{2s+1}$ on a homology class $x\in H_n\bigl(X;\Z_{(p)}\bigr)$ by the following formulas: 
	\[
	k^{2s+1}(x)=\{(\id\wedge k^{2s+1})\circ f \}\subset \pi_{n}\left( X\wedge\Sigma^{2s+1} K(\Z_{(p)})\right)\cong H_{n-(2s+1)}\bigl(X;\Z_{(p)}\bigr),
	\]
	\[
	\widetilde{k}^{2s+1}(x)=\{(\id\wedge \widetilde{k}^{2s+1})\circ f \}\subset \pi_{n}\left( X\wedge\Sigma^{2s} K(\mathbb{Q}/\Z_{(p)})\right)\cong H_{n-2s}\bigl(X;\mathbb{Q}/\Z_{(p)}\bigr),
	\]
	where $f\colon \mathbb{S}^n \to X \wedge \bu_{(p)}(s)$ runs over all maps such that the diagram
	\[
	\xymatrix{
		& X\wedge K(\Z_{(p)})  \\
		\mathbb{S}^n \ar[ur]^-{x} \ar[r]^-{f} & X \wedge \bu_{(p)}(s) \ar[u]_{\id \wedge u_s} 
	}
	\]
	is commutative and 
	\[
	\id\wedge k^{2s+1} \colon X \wedge \bu_{(p)}(s) \to   X\wedge\Sigma^{2s+1} K\bigl(\Z_{(p)}\bigr);
	\]
	\[
	\id\wedge \widetilde{k}^{2s+1} \colon X \wedge \bu_{(p)}(s) \to  X\wedge\Sigma^{2s} K\bigl(\mathbb{Q}/\Z_{(p)}\bigr).
	\]

	By analogy, $k^{2s+1}(x)$ is either empty or a coset of $k^{2s+1}(0)$, and  $\widetilde{k}^{2s+1}(x)$ is either empty or a coset of $\widetilde{k}^{2s+1}(0)$, where $0\in H_n\bigl(X;\Z_{(p)}\bigr)$. Denote the subgroup $k^{2s+1}(0)$ by $\Indet_n(k^{2s+1})$, and the subgroup~$\widetilde{k}^{2s+1}(0)$ by $\Indet_n(\widetilde{k}^{2s+1})$.
	
	If $X$ is a CW complex, then we define the action of $k^{2s+1}$ and $\widetilde{k}^{2s+1}$ on its (co)homology classes in the standard way by using the suspension spectrum of~$X$.   
	\end{definition}
 
Let $\Phi$ be one of the operations $k^{2s+1}$ or $\widetilde{k}^{2s+1}$. We say that $\Phi$ is \textit{defined} on a class $x$ if the set $\Phi(x)$ is not empty.  We say that $\Phi$ \textit{acts trivially} on a class $x$  if $\Phi$ is defined on $x$ and $0 \in \Phi(x)$. The following properties of $\Phi$ follow directly from the definition.
 \begin{description}
	\item[Stability] for any $\Z_{(p)}$-(co)homology class $x$ 
	$$\Phi(\Sigma x)=(-1)^{\deg\Phi}\Sigma \Phi(x).$$
	\item[Functoriality] if $x\in H_*\bigl(X;\Z_{(p)}\bigr)$, $y\in H^*\bigl(Y;\Z_{(p)}\bigr)$ and $f\colon X\to Y$, then
	$$\Phi(f^*y)\supset f^*\Phi(y),$$
	$$\Phi(f_*x)\supset f_*\Phi(x).$$
\end{description}
 
\begin{Rem}\label{bokstein and tilda for action of k-invariants}
The following properties of $k^{2s+1}$ and $\widetilde{k}^{2s+1}$ follow directly from the definition.
\begin{enumerate}
	\item The domains of $k^{2s+1}$ and $\widetilde{k}^{2s+1}$ coincide. 
	\item For any $\Z_{(p)}$-(co)homology class $x$ there is equality of sets $k^{2s+1}(x)=\beta\widetilde{k}^{2s+1}(x)$.
	\item For any dimension $n$ there are equalities of subgroups $\Indet^n(k^{2s+1})=\beta\Indet^n(\widetilde{k}^{2s+1})$ and $\Indet_n(k^{2s+1})=\beta\Indet_n(\widetilde{k}^{2s+1})$.
\end{enumerate}
\end{Rem}
 
In what follows, we are only interested in the case of finite spectra. However, all the following assertions can easily be generalized to the case of spectra which are bounded below and have finitely generated homotopy groups. 

Let $X$ be a finite spectrum. Then we can give an alternative definition of an action of $k^{2s+1}$ and~$\widetilde{k}^{2s+1}$ on homology classes of $X$ using the Spanier---Whitehead duality and the constructed action on cohomology. Namely, for a homology class $x\in H_n\bigl(X;\Z_{(p)}\bigr)$ we can define:
\[
	k^{2s+1}(x)=
	D^{-1}\Bigl(k^{2s+1}\bigl(D(x)\bigr)\Bigr)\subset H_{n-2s-1}\bigl(X;\Z_{(p)}\bigr),
	\]
	\[
	\widetilde{k}^{2s+1}(x)=
	D^{-1}\Bigl(\widetilde{k}^{2s+1}\bigl(D(x)\bigr)\Bigr)\subset H_{n-2s}\bigl(X;\mathbb{Q}/\Z_{(p)}\bigr).
\]
It can easily be verified by standard methods that the two definitions coincide, that is they give the same subsets in homology groups of $X$.  Henceforth we will mainly use the latter definition.  

Since for finite spectra one has the Spanier---Whitehead duality, we will formulate and prove most of the statements below only for cohomology classes, keeping in mind that analogous statements hold for homology classes as well. The only important difference is that the action of a (higher) cohomological operation on a cohomology class increases its degree while on a homology class it decreases degree. 
 
 Denote by $\K^*$ the complex $\K$-theory. Let us remind that the AHSS for the cohomology theory $\K^*\otimes \Z_{(p)}$ of a finite spectrum $X$ looks as follows: 
 \[
 E_2^{s,t}=H^s\bigl(X; \K^t(\pt)\otimes \Z_{(p)}\bigr) \Longrightarrow \K^*(X) \otimes \Z_{(p)} 
 \]
with differentials $d_r^{s,t}\colon E_r^{s,t} \to E_r^{s+r,t-r+1}$. Since $\K^*(\pt)\otimes \Z_{(p)}=\Z_{(p)}[h,h^{-1}]$ with $\deg(h)=2$, then 
\[
E_2^{s,2t}=H^s(X;\Z_{(p)}), \ \  E_2^{s,2t+1}=0
\]
and all differentials $d_{2r}^{*,*}$ with even indices vanish. 

The operations $k^{2s+1}$  are directly related to the differentials of the AHSS for $\K^*\otimes \Z_{(p)}$. Namely, there is the following theorem.

\begin{Th}[C.\,R.\,F. Maunder, \cite{Maunder}]\label{Maunder} 
	The differential $d_{2s+1}$ of the AHSS for $\K^*\otimes \Z_{(p)}$ is induced by the operation $k^{2s+1}$. Namely, for a finite spectrum $X$ the following statements hold. 
	  \begin{enumerate}
		\item The differential $d_{2s+1}^{n,2t}$ is defined on a class $x\in E_2^{n,2t}=H^n\bigl(X;\Z_{(p)}\bigr)$ if and only if the operation $k^{2s+1}$ is defined on $x$.
		\item Let a class $\tilde{x}\in E_{2s+1}^{n,2t}$ be represented by a class $x\in E_2^{n,2t}=H^n\bigl(X;\Z_{(p)}\bigr)$. Then there is the following equality of subsets of $E_2^{n+2s+1,2t-2s}=H^{n+2s+1}(X;\Z_{(p)})$:
		\[
		d_{2s+1}^{n,2t}(\tilde{x})=k^{2s+1}(x),
		\]
		where $d_{2s+1}^{n,2t}(\tilde{x})$ is considered as a coset of the union of images of the previous differentials. 
	\end{enumerate}
\end{Th}

 In \cite{Buch.1} and \cite{Buch.2} Buchstaber proved the following properties of the AHSS for  $\K^*\otimes \Z_{(p)}$. 
 \begin{Th}[{\cite[Lemma 1]{Buch.1}}]\label{Buch.1}
	In the AHSS for the theory $\K^*\otimes \Z_{(p)}$  the differential $d_{2s+1}$ vanishes, unless $s\equiv 0 \bmod (p-1).$
\end{Th}

\begin{Th}[{\cite[Theorem 2.1]{Buch.2}}]\label{Buch.2}
        For any spectrum $X$ and any  $r$, $n$, and $m$, in the AHSS for the theory $\K^*\otimes \Z_{(p)}$ the following formula holds. 
	\[
	d_{2r(p-1)+1}^{n,2m}(p^{r-1}x)=\varepsilon_r\bigl(\tilde{\beta} P^r(x)\bigr)\alpha\ \ \ (\text{modulo images of the previous differentials}),
	\]
	where $x\in E_2^{n,2m}, \alpha \in \K^{2m-2r(p-1)}(\pt)\otimes\Z_{(p)}$ is the generator, $\varepsilon_r$ is a unit in $\Z_{(p)}$, $P^r$ is the Steenrod cyclic reduced power for $p>2$ and $P^r=Sq^{2r}$ for $p=2$.
\end{Th}
 
Theorems \ref{Maunder} and \ref{Buch.1} imply the following statement. 
\begin{Cor}\label{1000010100} 
        If the operation $k^{2s+1}$ is defined on a class $x\in H^n\bigl(X;\Z_{(p)}\bigr)$, then the operation $k^{2s+1}$ acts trivially on $x$, unless $s\equiv 0 \bmod (p-1)$ . 
\end{Cor}
Thus only operations $k^{2r(p-1)+1}$ can act nontrivially. Denote the operation $k^{2r(p-1)+1}$ by~$\Phi_r$. 
\begin{Cor} \label{definition of Phi_r} 
	The following properties of the operations $\Phi_r$ hold. 
	\begin{enumerate}
		\item The operation $\Phi_r$ is defined on a class $x$ if and only if $\Phi_{r-1}$ acts trivially on $x$. The indeterminacy of $\Phi_r$ is equal to the union of images of $\Phi_{r-1}$ on the classes of the corresponding dimension.
		\item If $\Phi_r$ is defined on cohomology classes $x\in H^n\bigl(X;\Z_{(p)}\bigr)$ and $y\in H^m\bigl(Y;\Z_{(p)}\bigr)$, then $\Phi_r$  is defined on the class $x\times y\in H^{n+m}\bigl(X\wedge Y;\Z_{(p)}\bigr)$ and 
		$$\Phi_r(x\times y)\supset\Phi_r(x)\times y + (-1)^{n}x\times \Phi_r(y).$$ 
		\item For any cohomology class $x\in H^n\bigl(X;\Z_{(p)}\bigr)$ the operation $\Phi_r$ is defined on the class $p^{r-1}x$ and 
		$$\varepsilon_r\tilde{\beta} P^r(x)\in \Phi_r(p^{r-1}x),$$ 
		where $\varepsilon_r\in \Z_{(p)}$ is the same invertible element as in Theorem~\ref{Buch.2}. 	
         \end{enumerate} 
\end{Cor} 
\begin{proof}
	The first statement follows from Theorem~\ref{Maunder} and Corollary~\ref{1000010100}. The second statement follows from Theorem~\ref{Maunder} and the multiplicativity of the AHSS for $\K^*\otimes \Z_{(p)}$. The last statement follows from Theorem~\ref{Maunder} and Theorem~\ref{Buch.2}. 
\end{proof}

\subsection{Duality for \texorpdfstring{$\widetilde{\Phi}_r$}{Pr}}

Denote the operation $\widetilde{k}^{2r(p-1)}$ by $\widetilde{\Phi}_r$. By remark \ref{bokstein and tilda for action of k-invariants} we see that the domain of $\Phi_r$ is equal to the domain of $\widetilde{\Phi}_r$, for a (co)homology class $x$ there is the equality of sets $\Phi_r(x)=\beta\widetilde{\Phi}_r(x)$, and there are equalities of groups $\Indet^n(\Phi_r)=\beta\Indet^n(\widetilde{\Phi}_r)$, $\Indet_n(\Phi_r)=\beta\Indet_n(\widetilde{\Phi}_r)$.

\begin{Lem}\label{properties of Phi_r on Q-free cohomology}
	 Let $X$ be a finite spectrum such that $H_{\le N}(X;\mathbb{Q})=0$ for an integer $N$. Then for any $n<N$ the following statements hold. 
\begin{enumerate}
    \item The Bockstein homomorphism $\beta$ is an isomorphism. 
    \item The operation $\Phi_r$ acts trivially on a class $x\in H^n(X;\Z_{(p)})$ if and only if $\widetilde{\Phi}_r$ acts trivially on $x$.
    \item The indeterminacy $\Indet^{n-2r(p-1)}\bigl(\widetilde{\Phi}_r\bigr)$ is equal to the union of images of the operation $\widetilde{\Phi}_{r-1}$ on the cohomology classes of dimension $n-2(r-1)(p-1)$. 
    \end{enumerate}
\end{Lem}
\begin{proof}
Consider the long exact sequence of cohomology groups associated with the short exact sequence 
\[
0\to \Z_{(p)} \to \mathbb{Q} \to \mathbb{Q}/\Z_{(p)} \to 0
\]
of coefficient groups. Since $H_{\le N}(X;\mathbb{Q})=0$, then the Bockstein homomorphism $\beta$ is an isomorphism in dimensions less than $N$. This implies statement $1$. Since $\Phi_r(x)=\beta\widetilde{\Phi}_r(x)$ and $\beta$ is an isomorphism, then $0\in \Phi_r(x)$ if and only if  $0\in \widetilde{\Phi}_r(x)$. This implies statement $2$. Since  $\Phi_{r-1}(x)=\beta\widetilde{\Phi}_{r-1}(x)$, $\Indet^n(\Phi_r)=\beta\Indet^n(\widetilde{\Phi}_r)$, and $\beta$ is an isomorphism, then statement $1$ of Corollary~\ref{definition of Phi_r} implies statement $3$. 
\end{proof}

 For any classes $x\in H_n(X; G)$ and $y\in H^n(X; H)$, their Kronecker pairing $\langle x,y \rangle \in G\otimes_{\Z} H$ is defined. Since $\Z_{(p)}\otimes_{\Z}\bigl(\mathbb{Q}/\Z_{(p)}\bigr)=\mathbb{Q}/\Z_{(p)}$, then the value of the pairing of a class with $\Z_{(p)}$-coefficients and  a class with $\mathbb{Q}/\Z_{(p)}$-coefficients lies in $\mathbb{Q}/\Z_{(p)}$.

\begin{Lem}\label{tech for duality}
 Let $X$ be a finite spectrum such that $H_{\le N}(X;\mathbb{Q})=0$. Then for any $n<N$ the following statements hold. 
\begin{enumerate}
	\item The pairing $\langle \  ,\  \rangle \colon H_n(X;\Z_{(p)})\otimes H^n(X;\Z_{(p)})\to \Z_{(p)}$ is trivial. 
	\item The pairing $\langle \ , \  \rangle$ of $\Z_{(p)}$-classes with $\mathbb{Q}/\Z_{(p)}$-classes induces isomorphisms:
	$$H^n\bigl(X;\mathbb{Q}/\Z_{(p)}\bigr)\cong  \Hom\bigl(H_n(X;\Z_{(p)}),\mathbb{Q}/\Z_{(p)}\bigr),$$
	$$H_n\bigl(X;\Z_{(p)}\bigr)\cong  \Hom\bigl(H^n(X;\mathbb{Q}/\Z_{(p)}),\mathbb{Q}/\Z_{(p)}\bigr).$$
	\item For a subgroup $A \subset H^n\bigl(X;\mathbb{Q}/\Z_{(p)}\bigr)$ and a class $x\in H^n\bigl(X;\mathbb{Q}/\Z_{(p)}\bigr)$ such that $x \notin A$, there exists $z\in H_n\bigl(X; \Z_{(p)}\bigr)$ such that $\langle z , x \rangle\ne 0$ and $\langle z , a \rangle=0$ for all $a\in A$. 
\end{enumerate}
\end{Lem}
\begin{proof}
   Since $X$ is finite and $H^n(X;\mathbb{Q})=0$, then both $H_n(X;\Z_{(p)})$ and $H^n(X;\Z_{(p)})$ are finite abelian \linebreak $p$-groups. Since $\Z_{(p)}$ does not contain $p$-torsion, then for any abelian $p$-group $G$ we have that \linebreak $\Hom(G,\Z_{(p)})=0$. In particular, $\Hom(H^n(X;\Z_{(p)}),\Z_{(p)})=\Hom(H_n(X;\Z_{(p)}),\Z_{(p)})=0$. This implies statement $1$. 
   
   The universal coefficient theorem yields the short exact sequence
 $$
 0\to \Ext\bigl(H_{n-1}(X;\Z_{(p)}), \mathbb{Q}/\Z_{(p)}\bigr) \to H^n\bigl(X;\mathbb{Q}/\Z_{(p)}\bigr) \to \Hom\bigl(H_n(X;\Z_{(p)}),\mathbb{Q}/\Z_{(p)}\bigr) \to 0. 
 $$
  Since $\mathbb{Q}/\Z_{(p)}$ is divisible, then $\Ext\bigl(H_{n-1}(X;\Z_{(p)}), \mathbb{Q}/\Z_{(p)}\bigr)=0$ and therefore
  \[
  	 H^n\bigl(X;\mathbb{Q}/\Z_{(p)}\bigr) \cong \Hom\bigl(H_n(X;\Z_{(p)}),\mathbb{Q}/\Z_{(p)}\bigr).
\]
    Now let us proof that  $H_n\bigl(X;\Z_{(p)}\bigr)\cong  \Hom\bigl(H^n(X;\mathbb{Q}/\Z_{(p)}),\mathbb{Q}/\Z_{(p)}\bigr)$. Recall that the Pontryagin dual of an abelian group $G$ is the group $\widehat{G}= \Hom(G, \mathrm{S}^1)$. Note that for a finite abelian $p$-group $G$ there is an isomorphism $\widehat{G}\cong \Hom(G, \mathbb{Q}/\Z_{(p)})$. Therefore, the Pontryagin duality theorem yields that the map 
 $$
H_n(X;\Z_{(p)}) \to \Hom\Bigl( \Hom\bigl(H_n(X;\Z_{(p)}),\mathbb{Q}/\Z_{(p)}\bigr), \mathbb{Q}/\Z_{(p)} \Bigr)
 $$
is an isomorphism. Thus 
  $$ \Hom\bigl(H^n(X;\mathbb{Q}/\Z_{(p)}),\mathbb{Q}/\Z_{(p)}\bigr) \cong \Hom\Bigl( \Hom\bigl(H_n(X;\Z_{(p)}),\mathbb{Q}/\Z_{(p)}\bigr), \mathbb{Q}/\Z_{(p)} \Bigr) \cong H_n(X;\Z_{(p)}).$$
 This concludes the proof of statement $2$.
 
 Denote by $j\colon H^n\bigl(X;\mathbb{Q}/\Z_{(p)}\bigr) \to H^n\bigl(X;\mathbb{Q}/\Z_{(p)}\bigr)/A$ the quotient map. Since $x\notin A$, then $j(x)\ne 0$. Since $H^n(X;\mathbb{Q}/\Z_{(p)})/A$ is a finite abelian $p$-group, then there exists 
 \[
 f\in \Hom\bigl(H^n(X;\mathbb{Q}/\Z_{(p)})/A , \mathbb{Q}/\Z_{(p)}\bigr)
 \]
 such that $f\bigl(j(x)\bigr)\ne 0$. So we can take $z$ to be the image of $f\circ j$ under the isomorphism 
\[
 \Hom\bigl(H^n(X;\mathbb{Q}/\Z_{(p)}), \mathbb{Q}/\Z_{(p)}\bigr) \cong H_n\bigl(X; \Z_{(p)}\bigr). 
 \tag*{\qedhere}
 \]
\end{proof}

 \begin{Th}\label{duality}
	Let $X$ be a finite spectrum such that $H_{\le N}(X;\mathbb{Q})=0$ for an integer $N$. Then for any $n<N-2r(p-1)$ the following statements hold. 
	\begin{enumerate}
		\item[1a.] For a cohomology class $x\in H^n\bigl(X;\Z_{(p)}\bigr)$, the following two conditions are equivalent:
		\begin{enumerate}[label=(\roman*)]
			\item $\Phi_r$ is defined on $x$;
			\item $\langle \Indet_{n+2r(p-1)}(\widetilde{\Phi}_r), x \rangle  =0$.
		\end{enumerate} 
	        \item[1b.] For a homology class $x\in H_{n+2r(p-1)}\bigl(X;\Z_{(p)}\bigr)$, the following two conditions are equivalent:
		\begin{enumerate}[label=(\roman*)]
			\item $\Phi_r$ is defined on $x$;
			\item $\langle x,  \Indet^{n}(\widetilde{\Phi}_r) \rangle  =0$.
		\end{enumerate} 
		\item[2.] Suppose that $\Phi_r$ is defined on classes $x\in H^n\bigl(X;\Z_{(p)}\bigr)$ and $y\in H_{n+2r(p-1)}\bigl(X;\Z_{(p)}\bigr)$.  Then 
		$$\langle y, \widetilde{\Phi}_r(x)\rangle  =-\langle \widetilde{\Phi}_r(y), x \rangle.$$
		The pairings in both sides of this equality are well-defined elements of $\mathbb{Q}/\Z_{(p)}$, since by statements $1a, 1b$ we have that $\langle \Indet_{n+2r(p-1)}(\widetilde{\Phi}_r), x \rangle  =0$ and $\langle y, \Indet^{n}(\widetilde{\Phi}_r)\rangle  =0$. 
	\end{enumerate}
\end{Th}

\begin{proof} 
We will prove the theorem by induction on $r$. 
	
	 The base step: $r=1$.
	\\ Theorem~\ref{Buch.2} implies that $\Phi_1=\varepsilon_1\tilde{\beta} P^1$ is an ordinary cohomological operation, so statements $1a, 1b$ automatically hold. Now let us prove statement $2$. Since $\beta$ is an isomorphism and $\tilde{\beta}=\beta\rho_{*}$, then $\widetilde{\Phi}_{1}(x)=\varepsilon_1\rho_{*}P^1(x)$. Similarly, $\widetilde{\Phi}_{1}(y)=\varepsilon_1\rho_{*}P^1(y)$. Using the map $\eta \colon \mathbb{S} \to D(X)\wedge X$, we obtain that  
\begin{multline*}
\langle y, \widetilde{\Phi}_1(x) \rangle=\varepsilon_1\eta^*\bigl(D(y)\times \rho_{*}P^1(x)\bigr)=\varepsilon_1\eta^*\Bigl(\rho_*P^1\bigl(D(y)\times x\bigr)- \rho_*P^1\bigl(D(y)\bigr)\times x\Bigr)=\\
=\varepsilon_1 \rho_*P^1\eta^*\Bigl(D(y)\times x\Bigr)-\varepsilon_1\eta^*\Bigl(D\bigl(D^{-1}\rho_*P^1(D(y)\bigr)\times x\Bigr)=- \langle \widetilde{\Phi}_1(y), x \rangle.
\end{multline*}
The last equality holds since $\eta^*\bigl(D(y)\times x\bigr)\in H^{0}(\mathbb{S}; \Z_{(p)})$ and hence $P^1\eta^*\bigl(D(y)\times x \bigr)=0$.

The induction step: $r-1 \mapsto r$. \\
	Firstly, let us prove statement $1a$. Statement $1$ of Corollary \ref{definition of Phi_r} implies that $\Phi_r$ is defined on $x$ iff $\Phi_{r-1}$ acts on $x$ trivially. Statement $2$ of Lemma  \ref{properties of Phi_r on Q-free cohomology} implies that $\Phi_{r-1}$ acts on $x$ trivially iff $\widetilde{\Phi}_{r-1}$ acts on $x$ trivially. Statement $3$ of Lemma \ref{tech for duality} yields that $\widetilde{\Phi}_{r-1}$ acts on $x$ trivially iff for any $y\in H_{n+2(r-1)(p-1)}(X;\Z_{(p)})$ such that $\langle y, \Indet^{n}(\widetilde{\Phi}_{r-1})\rangle=0$ we have that $\langle y, \widetilde{\Phi}_{r-1}(x)\rangle=0$. By the induction assumption the last condition is equivalent to the condition that for any $y\in H_{n+2(r-1)(p-1)}(X;\Z_{(p)})$ such that $\Phi_{r-1}$ is defined on $y$ we have that $\langle \widetilde{\Phi}_{r-1}(y), x \rangle =0$. But statement $3$ of Lemma \ref{properties of Phi_r on Q-free cohomology} implies that $\Indet_{n+2r(p-1)}(\widetilde{\Phi}_r)$ is equal to the union of images of the operation $\widetilde{\Phi}_{r-1}$ on the homology classes of dimension $n+2(r-1)(p-1)$. This concludes the proof of statement  $1a$. Statement  $1b$ can be proved analogously. 
	
	Now let us prove statement $2$. For simplicity we will write $u, k$ and $\tilde{k}$ instead of $u_{r(p-1)}, k^{2r(p-1)+1}$ and $\tilde{k}^{2r(p-1)+1}$, respectively. Since the theory $\K^*\otimes \Z_{(p)}$ is multiplicative one has a canonical map 
	$$
	\mu\colon \bu_{(p)}\bigl(r(p-1)\bigr)\wedge \bu_{(p)}\bigl(r(p-1)\bigr) \to  \bu_{(p)}\bigl(r(p-1)\bigr),
	$$
	and besides $\mu^*(u)=u\times u$. Theorem  $3.2$ of \cite{Kahn} implies that 
	$$\mu^*(k)=k\times u+u\times k.$$
	Since $\beta\tilde{k}=k$, we obtain that 
	\begin{equation}\label{111}
	\mu^*(\tilde{k})=\tilde{k}\times u+u\times \tilde{k}+A,
	\end{equation}
	where the cohomology class $A$ lies in the kernel of $\beta$ and hence is the image of a class with rational coefficients. 
	
	For a pair of maps 
	$$f\colon X\to \bu_{(p)}\bigl(r(p-1)\bigr)\wedge \mathbb{S}^n  \ \ \text{and} \ \ g\colon D(X)\to \bu_{(p)}\bigl(r(p-1)\bigr)\wedge \mathbb{S}^{-n-2r(p-1)}$$
	such that 
	$$f^*\bigl(u\times \sigma^n \bigr)=x  \ \ \text{and} \ \ g^*\bigl(u\times \sigma^{-n-2r(p-1)} \bigr)=D(y)$$ 
	consider the map 
	$$
	h=(\mu\wedge \id)\circ c \circ (g\wedge f)\colon D(X)\wedge X \to \bu_{(p)}\bigl(r(p-1)\bigr)\wedge \mathbb{S}^{-2r(p-1)},
	$$	
	where 
	$$
	c\colon  \bu_{(p)}\bigl(r(p-1)\bigr)\wedge  \mathbb{S}^{-n-2r(p-1)} \wedge \bu_{(p)}\bigl(r(p-1)\bigr)\wedge \mathbb{S}^{n}  \to \bu_{(p)}\bigl(r(p-1)\bigr)\wedge \bu_{(p)}\bigl(r(p-1)\bigr) \wedge \mathbb{S}^{-2r(p-1)} 
	$$
	is the canonical homotopy equivalence. Thus we have that 
	$$h^*(u \times \sigma^{-2r(p-1)})=D(y)\times x.$$	
	So using equality \eqref{111}, we obtain that 
	\begin{multline}\label{3974}
	h^*(\tilde{k}\times \sigma^{-2r(p-1)})=\bigl(c\circ (g\wedge f) \bigr)^*\Bigl(\tilde{k} \times u\times \sigma^{-2r(p-1)} + u\times \tilde{k} 
	\times \sigma^{-2r(p-1)}+ A\times \sigma^{-2r(p-1)}\Bigr)
	=\\= g^*(\tilde{k} \times
	 \sigma^{-n-2r(p-1)})\times x + D(y)\times f^*(\tilde{k}\times \sigma^n) + B
	\end{multline}
	where $B=\bigl(c\circ(g\wedge f)\bigr)^*(A\times\sigma^{-2r(p-1)})$.
	
	Now, let us prove that $B=0$. Recall that the rational cohomology groups of the spectrum $D(X)$ are trivial in dimensions  $\ge -N$ and the rational cohomology groups of the spectrum $$\bu_{(p)}\bigl(r(p-1)\bigr)\wedge \mathbb{S}^{-n-2r(p-1)}$$ are trivial in dimensions $< -n-2r(p-1)$.  Since $n+2r(p-1)< N$, we get  that the map
\[
g\colon D(X) \to \bu_{(p)}\bigl(r(p-1)\bigr)\wedge \mathbb{S}^{-n-2r(p-1)}
\]
induces zero homomorphism in cohomology with rational coefficients. Therefore, the map $g\wedge f$ induces zero homomorphism in cohomology with rational coefficients. Since $A$ is the image of a class with rational coefficients, we obtain that  $B=0$. 
	
	Thus, formula \eqref{3974} implies that 
	\begin{equation}\label{34345}
	\widetilde{\Phi}_r(D(y)\times x)\supset \widetilde{\Phi}_r(D(y))\times x + D(y)\times \widetilde{\Phi}_r(x).
	\end{equation}
	Consider the map $\eta \colon \mathbb{S} \to D(X)\wedge X$. By the functoriality of  $\widetilde{\Phi}_r$ we have that
$$\eta^*\Bigl(\widetilde{\Phi}_r\bigl(D(y)\times x\bigr)\Bigr)\subset\widetilde{\Phi}_r\Bigl(\eta^*\bigl(D(y)\times x\bigr)\Bigr).$$
	Since $\eta^*(D(y)\times x)\in H^0(\mathbb{S};\Z_{(p)})$, then $\widetilde{\Phi}_r\Bigl(\eta^*\bigl(D(y)\times x\bigr)\Bigr)=0$. Therefore, formula \eqref{34345} implies that 
	\[
	 	\eta^*\Bigl(D(y)\times \widetilde{\Phi}_r\bigl(x\bigr)\Bigr)+\eta^*\Bigl(\widetilde{\Phi}_r\bigl(D(y)\bigr)\times x\Bigr)=0
	\]
		Finally, since $\widetilde{\Phi}_r\bigl(D(y)\bigr)=D\bigl(\widetilde{\Phi}_r(y)\bigr)$, we have that 
		\[
		\langle y, \widetilde{\Phi}_r(x)\rangle=\eta^*\bigl(D(y)\times \widetilde{\Phi}_r(x)\bigr)=-\eta^*\bigl(\widetilde{\Phi}_r(D(y))\times x\bigr)=-\eta^*\bigl(D(\widetilde{\Phi}_r(y))\times x\bigr)=-\langle \widetilde{\Phi}_r(y), x \rangle. 
\tag*{\qedhere}		
\]
\end{proof}
 
\subsection{Application to the problem of realization of cycles}
Denote by $\omega\in H^0(\MU;\Z)$ the universal Thom class. The map $\MU\to K(\Z)$ classifying the class $\omega$ induces natural transformations of the corresponding homology and cohomology theories:
\[
\omega^*\colon \MU^n(X)\to H^n(X;\Z),
\ \ \ \ \
\omega_*\colon \MU_n(X)\to H_n(X;\Z).
\]
Homology and cohomology classes in the image of these maps will be called $\U$-realizable. Note that a homology class $x\in H_n(X;\Z)$ of a CW complex $X$ is $\U$-realizable if and only if there exist a smooth closed stably complex manifold $M$ and a map $f\colon M \to X$ such that $x=f_*[M]$. 

If $X$ is a finite spectrum, then there is the following commutative diagram	
\[
\xymatrix{
	\MU_n(X)\ar[r]^-{D} \ar[d]^{\omega_*} & \MU^{-n}\bigl(D(X)\bigr) \ar[d]^{\omega^*} \\
	H_n(X;\Z)\ar[r]^-{D}  & H^{-n}\bigl(D(X);\Z\bigr).	
}
\]
Thus a homology class $a\in H_n(X;\Z)$ of a finite spectrum $X$ is $\U$-realizable if and only if the cohomology class $D(a)\in H^{-n}\bigl(D(X);\Z\bigr)$ is $\U$-realizable. 
\begin{Lem}\label{hiegher operations and realiazation}
	If a homology class $a\in H_n(X;\Z)$ of a finite CW complex $X$ is $\U$-realizable, then for any prime $p$ and any positive integer $r$ the operation $\Phi_r$ acts trivially on $a$.   
\end{Lem}
\begin{proof}
Let us replace $X$ by its suspension spectrum. Since $a$ is $\U$-realizable, then $D(a)$ is $\U$-realizable and so there is a map $f\colon D(X)\to \Sigma^{-n}\MU$ such that $f^*(\Sigma^{-n}\omega)=D(a)$. Note that if $\Phi_r$ is defined on $\omega$, then $\Phi_r$ acts trivially on  $\omega$, since the cohomology groups $H^*(\MU;\Z)$ are torsion free and the operation~$\Phi_r$ has finite order. By Corollary \ref{definition of Phi_r}, the operation $\Phi_1$ is defined on any cohomology class and if $\Phi_{r-1}$ acts trivially on a class $x$, then $\Phi_r$ is defined on $x$. So by induction on $r$ we obtain that all the operations $\Phi_{r}$ act trivially on the class $\omega$ and thus on $D(a)$. Since, by definition, $\Phi_r(a)=D^{-1}\bigl(\Phi_r(D(a))\bigr)$, then all the operations $\Phi_{r}$ act trivially on $a$. 
\end{proof}
\begin{Th}\label{Lem for example} 
	Let $X$ be a $CW$ complex such that the following conditions hold.
	\begin{enumerate}
		\item $\widetilde{H}_*(X;\mathbb{Q})=0.$
		\item There exists a class $x\in H^{n}\bigl(X;\Z_{(p)}\bigr)$ such that $\Phi_r$ is defined on $x$ and $\tilde{\beta}\chi(P^r)(x)\ne0$. 
	\end{enumerate}
	Then for any $a\in H_{n+2r(p-1)}(X;\Z)$ satisfying $\langle a,\chi(P^r)(x)\rangle  \ne0$ and any integer $m$ coprime with $p$, the class $mp^{r-1}a$ is not $\U$-realizable. Moreover, such a class $a$ always exists. 
\end{Th}
\begin{proof}
Firstly, let us prove the existence of a class $a$. Since $\tilde{\beta}\chi(P^r)(x)\ne 0$, then $\rho_*\chi(P^r)(x)\ne 0$. As in the proof of statement $2$ of Lemma \ref{tech for duality}, by the universal coefficient theorem and the divisibility of $\Q/\Z_{(p)}$  we obtain that 
\[
H^*\bigl(X;\Q/\Z_{(p)}\bigr) \cong \Hom\bigl(H_*(X;\Z);\Q/\Z_{(p)}\bigl).
\]
Thus there exists a class $a\in H_{n+2r(p-1)}(X;\Z)$ such that $\langle a,\rho_*\chi(P^r)(x)\rangle  \ne0$. Since 
\[
\langle a,\rho_*\chi(P^r)(x)\rangle=\rho\langle a,\chi(P^r)(x)\rangle
\]
and  $\rho$ is injective, we obtain that $\langle a,\chi(P^r)(x)\rangle\ne 0$. 

Now let us prove that the class $mp^{r-1}a$ is not $\U$-realizable. Since the class $ma$ also satisfies the conditions of the lemma it is sufficient to prove that the class $p^{r-1}a$ is not $\U$-realizable. We will argue by contradiction. 

All the conditions and statements of the lemma do not change if we replace $X$ by $\Sigma X$. So we can assume that $X$ does not contain $1$-dimensional cells. In particular, $X$ is simply connected. 
	
	Let us consider a closed stably complex manifold $M$ and a map $f\colon M \to X$ such that $f_*[M]=p^{r-1}a$. Since $M$ is compact, then $f(M)$ is contained in a finite subcomplex $X'\subset X$. Let $X_1 \subset X$ be a finite subcomplex such that $X'\subset X_1$ and there exists $b'\in H_{n+2r(p-1)}(X_1; \Z)$ satisfying $(i_1)_*b'=a$, where $i_1\colon X_1\to X$ is the inclusion.  Since $X$ does not contain $1$-dimensional cells, then $X_1$ does not contain $1$-dimensional cells and in particular is simply connected. 
	
Since $X$ is simply connected and $\widetilde{H}_*(X;\mathbb{Q})=0$, then the rational Hurewicz theorem implies that $\pi_*(X)\otimes \Q=0$. Since $X_1$ is finite and simply connected, then the rational Hurewicz theorem implies that $\pi_2(X_1)\otimes \Q$ is finitely generated. Let 
\[
\alpha_l \colon S^2 \to X_1, \qquad 1\le l \le m
\]
 be spheroids such that they generate $\pi_2(X_1)\otimes \Q$ and the maps 
 \[
 i_1\circ \alpha_l\colon S^2 \to X, \qquad  1\le l \le m
 \]
are null-homotopic. Attaching $3$-cells to $X_1$ along the spheroids $\alpha_1, \dots, \alpha_m$, we get a finite CW complex~$X_2$. Let $j_2\colon X_1 \to X_2$ be the inclusion of the subcomplex. Since all the maps 
\[
i_1\circ \alpha_l\colon S^2 \to X, \ 1\le l \le m
\] 
are null-homotopic, then the map $i_1$ can be extended to a map $i_2\colon X_2 \to X$, so that $i_1=i_2\circ j_2$. By the construction, we obtained that $\pi_2(X_2)\otimes \Q=0$. Killing the next rational homotopy groups analogously, we get finite CW complexes $X_k$ such that $\pi_{\le k}(X_k)\otimes \Q=0$ and maps $i_k \colon X_k \to X$ and $j_k\colon X_1 \to X_k$ such that $i_1=i_k\circ j_k$.
	
	Denote by $Y$ the finite CW complex $X_N$ with $N> n+2r(p-1)$ and by $i$ and $j$ the maps $i_N$ and $j_N$ respectively. Then we have the following commutative diagram. 
	\[
	\xymatrix{
	Y  \ar[dr]^{i} &\\
	X_1 \ar[u]^{j} \ar[r]^{i_1}& X
	}
	\]
	Denote by $b\in H_{n+2r(p-1)}(Y; \Z)$ the homology class $j_*(b')$, and by $y\in H^n(Y;\Z_{(p)})$ the cohomology class~$i^*(x)$.
	
	By the construction of $X_1$ the class $p^{r-1}b'$ is $\U$-realizable. Therefore, the class $p^{r-1}b$  is also $\U$-realizable. Since $\langle a,\chi(P^r)(x)\rangle  \ne0$, then $\langle b,\chi(P^r)(y)\rangle\ne 0$. Since $\Phi_r$ is defined on $x$, then $\Phi_r$ is defined on $y=i^*(x)$.

	Thus, replacing $Y$ by its suspension spectrum, we get the spectrum $Y$ such that 
	\begin{enumerate}[label=(\roman*)]
	\item $Y$ is a finite CW spectrum such that $H_{\le N}(Y;\Q)=0$ and $N>n+2r(p-1)$;
	\item there is $y\in H^n(Y;\Z_{(p)})$ such that $\Phi_r$ is defined on $y$;
	\item there is $b\in H_{n+2r(p-1)}(Y; \Z)$ such that $\langle b,\chi(P^r)(y)\rangle\ne 0$;
	\item $p^{r-1}b$ is $\U$-realizable.
	\end{enumerate}
	
Now let us prove that properties (i), (ii) and (iii) imply that $\Phi_r$ is defined and acts non-trivially on~$p^{r-1}b$. Indeed, statement $3$ of Corollary  \ref{definition of Phi_r} implies that $\varepsilon_r\beta\rho_* P^r(b)\in \Phi_r(p^{r-1}b)$. Since $\beta$ is an isomorphism in this range of dimensions (statement $1$ of Lemma \ref{properties of Phi_r on Q-free cohomology}), then $\varepsilon_r\rho_* P^r(b)\in \widetilde{\Phi}_r(p^{r-1}b)$. Property~(ii) and statement $1a$ of Theorem~\ref{duality} imply that $\langle \Indet_{n+2r(p-1)}(\widetilde{\Phi}_r), y \rangle =0$. Thus we have the following equality  
\[
\langle \widetilde{\Phi}_r(p^{r-1}b), y \rangle  =\langle \varepsilon_r\rho_* P^r(b), y \rangle. 
\]
Since $\langle P^r(b), y \rangle  =\langle b, \chi(P^r)(y)\rangle$ (see \cite{Thom}, \S 3.4), we obtain that 
\[
\langle \varepsilon_r\rho_* P^r(b), y \rangle=\rho\bigl(\varepsilon_r \langle P^r(b), y \rangle\bigr)= \rho\bigl( \varepsilon_r\langle b, \chi(P^r)(y)\rangle\bigr).
\]
Since $\rho\colon \Z/p\Z \to \Q/\Z_{(p)}$ is injective, $\varepsilon_r\ne 0$ and $\langle b, \chi(P^r)(y)\rangle\ne 0$, we get that 
$
\rho\bigl( \varepsilon_r\langle b,\chi(P^r)(y)\rangle\bigr)\ne 0.
$ 
Therefore
\[
\langle \varepsilon_r\rho_* P^r(b), y \rangle\ne 0.
\]
Thus $\langle \widetilde{\Phi}_r(p^{r-1}b), y \rangle  =\langle \varepsilon_r\rho_* P^r(b), y \rangle\ne 0$. So $0 \notin \widetilde{\Phi}_r(p^{r-1}b)$. Finally, statement $2$ of Lemma \ref{properties of Phi_r on Q-free cohomology} implies that $\Phi_r$ acts non-trivially on $p^{r-1}b$. 

On the other hand, property (iv) and Lemma \ref{hiegher operations and realiazation} imply that all the operations $\Phi_r$  act trivially on $p^{r-1}b$. Contradiction. 	
 \end{proof}
 
 \section{Action of \texorpdfstring{$\chi(P^r)$}{x(Pr)}}
Let $\mathcal{A}_p$ denote the $\bmod \ p$ \ Steenrod algebra. Let $Sq^i\in\mathcal{A}_2$ and $P^i\in\mathcal{A}_p$ for $p>2$ be the Steenrod cyclic reduced powers. Let $\beta_p\in \mathcal{A}_p$ be the $\bmod~p$ Bockstein homomorphism. Let $\mathcal{A}_p/(\beta_p)$ denote the quotient of $\mathcal{A}_p$ by the two-sided ideal generated by $\beta_p$. Let $\chi$ denote the antipode of $\mathcal{A}_p$. For a Hopf algebra $\mathcal{H}$, we denote by $\mathcal{H}^*$ the dual Hopf algebra. 
	
Also we need the following notation. 	
\begin{itemize}	
	\item Let $\varGamma$ be the set of infinite series of nonnegative integers $(x_1,x_2,\dots)$ such that: 
	\begin{enumerate}[label=\alph*)]
		\item only finitely many of $x_i$ are nonzero;
		\item $x_i\ge px_{i+1}$ for all $i$.
	\end{enumerate}
	
	\item Let $\varUpsilon$ denote the set of infinite series of nonnegative integers $(x_1,x_2,\dots)$ such that only finitely many of $x_i$ are nonzero. 
	
	\item Let $\gamma \colon  \varGamma \to \varUpsilon$ be the map defined by the formula
	$$\gamma\bigl((x_1, x_2,x_3, \dots)\bigr)=(x_1-px_2,x_2-px_3,\dots).$$
	It is easy to see that $\gamma$ is a bijection with the inverse map $\gamma^{-1}\colon \varUpsilon \to \varGamma$ given by  
	\begin{align*}
	\gamma^{-1}\bigl((x_1,\dots,x_n, 0, \dots)\bigr)=
	(&x_1+px_2+p^2x_3+\ldots +p^{n-1}x_n,\\  &x_2+px_3+ \ldots +p^{n-2}x_n, \dots, x_{n-1}+px_{n},x_n, 0, \dots),
	\end{align*}
where $x_m=0$ for all $m>n$. 
	
\end{itemize}	
	
Endow the sets $\varGamma$ and $\varUpsilon$ with the right lexicographical order. Recall that in the right lexicographical order 
 $
 (x_1, x_2, \dots)>(y_1, y_2, \dots)
 $
 if there exists an index $i$ such that $x_i>y_i$ and $x_j=y_j$ for all $j>i$. It is readily verified that $\gamma$ preserves the order. 

Recall that one has Milnor's generators $\zeta_i\in \mathcal{A}_2^*$ and $\xi_i, \tau_i \in \mathcal{A}_p^*$ for  $p>2$ (see \cite{Milnor}). The following results are well known.
\begin{Prop}[\cite{Milnor}; {\cite[section 2.5]{Kochman}}]\label{structure of dual algebra}
	For $p=2$, the following isomorphisms hold:
	$$
	\mathcal{A}_2^* \cong \Z/2\Z[\zeta_1,\dots,\zeta_n,\dots], \ \dim(\zeta_i)=2^i-1;
	$$
	$$
	\bigl(\mathcal{A}_2/(\beta_2)\bigr)^* \cong \Z/2\Z[\zeta_{1}^2,\dots,\zeta_{n}^2,\dots].
	$$
	For $p>2$, the following isomorphisms hold:
	$$
	\mathcal{A}_p^* \cong \Z/p\Z[\xi_1,\dots,\xi_n,\dots]\otimes_{\Z/p\Z} \Lambda(\tau_0,\dots,\tau_n,\dots), \ \dim(\xi_i)=2(p^i-1), \ \dim(\tau_i)=2p^i-1;
	$$	
	$$
	\bigl(\mathcal{A}_p/(\beta_p)\bigr)^* \cong \Z/p\Z[\xi_1,\dots,\xi_n,\dots].
	$$
\end{Prop}

For a sequence $J=(x_1,x_2,\dots, x_n, 0,0, \dots)\in \varUpsilon$, we denote 
\[
\zeta^J=\zeta_1^{x_1}\zeta_2^{x_2}\dots\zeta_n^{x_n}, \quad \xi^J=\xi_1^{x_1}\xi_2^{x_2}\dots\xi_n^{x_n}, \quad \tau^J=\tau_0^{x_1}\tau_1^{x_2}\dots\tau_{n-1}^{x_n}.
\]
For $p=2$, Proposition \ref{structure of dual algebra} yields that the set $\{\zeta^J \colon J\in \varUpsilon\}$ is a basis of 
$\mathcal{A}_2^*$ and $\{\zeta^{2J} \colon J\in \varUpsilon\}$ is a basis of~$\bigl(\mathcal{A}_2/(\beta_2)\bigr)^*$. For 
$p>2$, Proposition \ref{structure of dual algebra} yields that the set $\{\xi^J\tau^{J'} \colon J, J' \in \varUpsilon\}$ is
 a basis of $\mathcal{A}_p^*$ and $\{\xi^J \colon J \in \varUpsilon\}$ is a basis of $\bigl(\mathcal{A}_p/(\beta_p)\bigr)^*$.  For dual bases of dual algebras (so-called Milnor's bases) we use the upper star notation.  

Denote by $\varUpsilon_r$ the set of all sequences $(x_1,x_2,\dots, x_n, 0, 0, \dots)\in \varUpsilon$ such that  
\begin{equation}\label{5555}
	x_1+x_2(1+p)+x_3(1+p+p^2)+\dots+x_n(1+p+\dots+p^{n-1})=r.
\end{equation}
 Denote by $\varGamma_r$ the set of all sequences $I=(x_1,x_2,\dots) \in\varGamma$ such that $|I|=\sum_{i=1}^{\infty}x_i=r$. The definition of~$\gamma$ implies that $\gamma$ maps bijectively  $\varGamma_r$ to $\varUpsilon_r$. 
 
\begin{Prop}[{\cite[\S 7, cor. 6]{Milnor}}]\label{Milnor}
	The following equalities hold in $\mathcal{A}_p$:
	$$\chi(P^i)=(-1)^i\sum_{J\in \varUpsilon_i} (\xi^J)^*, \quad p>2; $$
	$$\chi(Sq^i)=\sum_{J\in \varUpsilon_i} (\zeta^J)^*, \quad p=2.$$
\end{Prop}
\begin{Cor}\label{Cor1}
	The following equalities hold in $\mathcal{A}_p/(\beta_p)$:
	$$\chi(P^i)=(-1)^i\sum_{J\in \varUpsilon_i} (\xi^J)^*, \quad p>2; $$
	$$\chi(Sq^{2i})=\sum_{J\in \varUpsilon_i} (\zeta^{2J})^*, \quad p=2. $$
\end{Cor}
\begin{proof}
Consider the quotient map $\pi\colon \mathcal{A}_p \to \mathcal{A}_p/(\beta_p)$. Proposition \ref{structure of dual algebra} yields that for $p>2$ we have that 
	\[
	\pi\bigl( (\xi^J\tau^{J'})^*\bigr)=\left\{ \begin{aligned}
	&(\xi^J)^*  &\text{if} \  J'=0,\\
	&0  &\text{otherwise}.
	\end{aligned}
	\right.
	\] 
	And for $p=2$ we have that 
	\[
	\pi\bigl( (\zeta^J)^*\bigr)=\left\{ \begin{aligned}
	&(\zeta^J)^* & \text{if there exists} \ J'\in \varUpsilon \ \text{such that} \   J=2J' ,\\
	&0 & \text{otherwise}.
	\end{aligned}
	\right.
	\] 
Applying this to Proposition \ref{Milnor}, we get the statement of the corollary. 
\end{proof}
For a sequence $I=(x_1,x_2,\dots, x_n, 0,0, \dots)\in \varGamma$ let us denote
\[
\begin{aligned}
P^I&=P^{x_1}P^{x_2}\dots P^{x_n}  & \text{for} \  p>2, \\
Sq^I&=Sq^{x_1}Sq^{x_2}\dots Sq^{x_n}  & \text{for} \  p=2.
\end{aligned}
\]
\begin{Lem}[{\cite[Lemma 8]{Milnor}}]\label{Milnor2}
	For any $I,I'\in \varGamma$ and $p>2$ we have that 
\[
	\langle \xi^{\gamma(I')}, P^I\rangle=\left\{ 
	\begin{aligned}
	&\  0& \text{if}\  I<I',\\
	&\pm1& \text{if}\  I=I'.
	\end{aligned}
	\right.
\]
For any $I,I'\in \varGamma$ and $p=2$ we have that	
\[
	\langle \zeta^{\gamma(I')}, Sq^I\rangle=\left\{ 
	\begin{aligned}
	&0 & \text{if}\  I<I',\\
	&1 &  \text{if}\  I=I'.
	\end{aligned}
	\right.
\]

\end{Lem}
By Proposition $3.5$ of \cite{Steenepstein}, the set $\{ P^I: I\in \Gamma\}$ is a basis in $\mathcal{A}_p/(\beta_p)$ for $p>2$, and the set $\{Sq^{2I}: I\in \Gamma\}$ is a basis in  $\mathcal{A}_2/(\beta_2)$. These bases are called Cartan---Serre bases.   
Hereinafter, for $p=2$ we denote $Sq^{2i}$ by $P^i$ and $\zeta^{2i}$ by $\xi^{i}$. With this change in mind, all further statements hold for all primes~$p$.
\begin{Cor}\label{tech 100}
The operation $\chi(P^r)$ written in the Cartan---Serre basis contains with non-zero coefficient the basis element $P^{I^r_{\max}}$, where the sequence $I^{r}_{\max}$ is the greatest sequence in $\Gamma_r$.
\end{Cor}
\begin{proof}
     Lemma \ref{Milnor2} implies that the transformation matrix from the Milnor basis to the Cartan---Serre basis is upper triangular with $\pm1$ on the diagonal. Therefore, in the Cartan---Serre basis the element~$\sum_{I\in \varUpsilon_r} (\xi^I)^*$ contains the basis element $P^{I^r_{\max}}$ with coefficient $\pm1$. On the other hand, Corollary \ref{Cor1} implies that $\chi(P^r)=(-1)^r\sum_{I\in \varUpsilon_r} (\xi^I)^*$.  
\end{proof}

\begin{Lem}\label{tech for combinatorial trm}
A sequence $(x_1, x_2, \dots)\in \varUpsilon_r $ is the greatest sequence in $\varUpsilon_r$ if and only if it satisfies the following two conditions. 
\begin{enumerate}
		\item $x_i\le p$ \ for all  $i$.
		\item If $x_i=p$, then $x_j=0$ \ for all $j<i$.
	\end{enumerate}
Moreover, the greatest sequence in $\varUpsilon_r$ has the minimal sum of terms among all sequences in~$\varUpsilon_r$. 
\end{Lem}
\begin{proof}
Firstly, we check that if a sequence $J=(x_1,x_2,\dots)\in \varUpsilon_r$ satisfies the conditions, then $J$ is the greatest in~$\varUpsilon_r$. Let $K=(y_1,y_2,\dots)$ be a sequence in $\varUpsilon_r$. Let us prove that $J\ge K$ by contradiction. 
Assume that $K>J$. Then there exists $i$ such that $y_i>x_i$ and $y_j=x_j$ for all $j>i$. Since $J$ and $K$ both lie in $\varUpsilon_r$, then 
$$ x_1+x_2(1+p)+\dots+x_i(1+p+\dots+p^{i-1})=y_1+y_2(1+p)+\dots+y_i(1+p+\dots+p^{i-1}).$$
Since $y_i>x_i$, then we obtain that 
$$x_1+x_2(1+p)+\dots+x_{i-1}(1+p+\dots+p^{i-2})\ge (1+p+\dots+p^{i-1}).$$
Let $s$ be the minimal index such that  $x_{s+1}\ne 0$. Then $s\le i-2$ and the conditions of lemma imply that 
\begin{multline*}
		x_1+x_2(1+p)+\dots+x_{i-1}(1+p+\dots+p^{i-2})\le\\ \le p(1+p+\dots+p^s)+(p-1)(1+p+\dots+p^{s+1})+\dots+(p-1)(1+p+\dots+p^{i-2})=\\
		=1+p+\dots+p^{i-1}-(i-s-1)<1+p+\dots+p^{i-1},
\end{multline*}
which gives a contradiction. 

Now, let us prove that the greatest sequence satisfies the conditions of the lemma and has the minimal sum of terms. Let $J=(x_1, x_2, \dots)$ be a sequence in $\varUpsilon_r$. Suppose that $J$ does not satisfy the conditions. Then we can define a new sequence $J'=(x'_1, x'_2, \dots)$ in one of the following two ways. 
\begin{description}
	\item[1] If $x_i>p$ for some $i$, then: 
	\\$x'_i=x_i-(p+1)$;
	\\$x'_{i+1}=x_{i+1}+1$;
	\\$x'_{i-1}=x_{i-1}+p$;
	\\$x'_{j}=x_{j}$ for $j \not\in \{i, i+1, i-1\}$.
	\item[2] If $x_i=p$ and $x_{i-k}>0$ for some $i$ and $k>0$, then:
	\\ $x'_i=x_i-p$;
	\\ $x'_{i+1}=x_{i+1}+1$;
	\\ $x'_{i-k}=x_{i-k}-1$;
	\\$x'_{i-k-1}=x_{i-k-1}+p$;
	\\$x'_{j}=x_{j}$ for $j \not\in \{i, i+1, i-k, i-k-1\}$. 
\end{description}
It is readily verified that $J'\in \varUpsilon_r$, $J'>J$ and $|J'|\le |J|$. Thus, if $J$ is the greatest in $ \varUpsilon_r$, then $J$ satisfies the conditions of the lemma, since otherwise $J'>J$. So it remains to prove that the greatest sequence has a minimal sum of terms. Let $J$ be a sequence such that for any $\widetilde{J}>J$, we have that $|\widetilde{J}|\ge|J|$. Then~$J$ satisfies the conditions of the lemma (so $J$ is the greatest), since otherwise $J'>J$ and $|J'|\le |J|$. 
\end{proof}

For an operation $P^I$ where $I\in \varGamma$, we denote by $\ex(P^I)$ the excess of $P^I$. Recall that $\ex(P^I)$ is equal to the minimal dimension of a $\Z/p\Z$-cohomology class on which $P^I$ can act non-trivially, and we have the formula $\ex(P^I)=2|\gamma(I)|$ (see {\cite[\S 4.L, p.\,499-501]{Hatch}}). For a positive integer $r$, we define the number~$\ex(r)$ by the formula
\[
\ex(r)=\min_{I\in \varGamma_r}\{\ex(P^I)\}.
\]

\begin{Th} \label{combinatorial}
	The following statements hold. 
	\begin{enumerate}
		\item $\ex(P^{I^r_{\max}})=\ex(r)$;
		\item $\ex(r)-\ex(r-1)$ is equal to either $2$ or $-2(p-1)$.
	\end{enumerate}
\end{Th}
\begin{proof}
Note that 
\[
\ex(r)=\min_{ I\in \varGamma_r}\{\ex(P^I)\}= \min_{I\in \varGamma_r}\{2|\gamma(I)|\}=2 |\gamma(I^r_{\max})|=\ex(P^{I^r_{\max}}).
\]
The penultimate equality holds, since $\gamma$ is a bijection and preserves the order, and by Lemma \ref{tech for combinatorial trm} the sequence~$\gamma(I^r_{\max})$ has the minimal sum of terms. 
	
Now let us prove the second statement. Let 
$$
	\gamma(I^{r-1}_{\max})=(x_1,x_2,\dots,x_n,0,\dots).
$$
Then $(x_1,x_2,\dots, x_n,0, \dots)$ is the greatest sequence in $\varUpsilon_r$. So by Lemma \ref{tech for combinatorial trm} we have two cases.
	\\i) For all $i$ we have that $x_i<p$. Then Lemma \ref{tech for combinatorial trm} yields that  
	$$\gamma(I^r_{\max})=(x_1+1,x_2,\dots,x_n,0,\dots).$$
	ii) There is an $i$ such that $x_j=0$ for all $j<i$, $x_i=p$, and $x_k<p$ for all $k>i$. Then Lemma \ref{tech for combinatorial trm} implies that 
	$$\gamma(I^r_{\max})=(0,\dots,0,x_{i+1}+1,\dots,x_n,0,\dots).$$ 
	Thus, $\ex(r)-\ex(r-1)$ is equal to either $2$ (case i) or $-2(p-1)$ (case ii).
\end{proof}
In section $4$ we will need the following standard lemma. 
\begin{Lem}\label{tech for chi(P^r)}
	The operation $\chi(P^r)$ acts nontrivially on the class
	$$\underbrace{\iota\times\dots\times \iota}_{\ex(r)/2}\in H^{\ex(r)}\bigl(\underbrace{K(\Z,2)\times\dots\times K(\Z,2)}_{\ex(r)/2};\Z/p\Z\bigr),$$
	where $\iota\in H^2(K(\Z,2);\Z/p\Z)=\Z/p\Z$ is the generator. 
\end{Lem}
\begin{proof}
       Consider the left lexicographical order on the set of monomials in $\iota$: 
	$$\iota^{a_1}\times \iota^{a_2}\times \dots \times  \iota^{a_{\ex(r)/2}} > \iota^{b_1}\times \iota^{b_2}\times\dots\times  \iota^{b_{\ex(r)/2}},$$
	if the sequence $(a_1,a_2,\dots,a_{\ex(r)/2})$ is greater than $(b_1,b_2,\dots,b_{\ex(r)/2})$ in the left lexicographical order. 
	
	Let us write $\chi(P^r)$ in the  Cartan---Serre basis: 
	$$\chi(P^r)=\sum_{J\in \varUpsilon_r}\varepsilon_J P^{\gamma^{-1}\left(J\right)}.$$
	By the definition of $\ex(r)$ we have two cases:\\
	1) $2|J|=\ex(P^{\gamma^{-1}(J)})>\ex(r)$.\\
	Then
	$$P^{\gamma^{-1}\left(J\right)}(\underbrace{\iota\times\dots\times \iota}_{\ex(r)/2})=0,$$
	since the degree of $\underbrace{\iota\times\dots\times \iota}_{\ex(r)/2}$ is equal to $\ex(r)$.\\
	2)  $2|J|=\ex(P^{\gamma^{-1}(J)})=\ex(r)$.\\
	Let 
	$$J=(x_1,x_2,\dots,x_n,0,\dots).$$
	Let us write $P^{\gamma^{-1}\left(J\right)}(\underbrace{\iota\times\dots\times \iota}_{\ex(r)/2})$ as a sum of monomials in $\iota$. Then by straightforward computation it can be checked that this sum contains with coefficient $1$ the monomial 
	\begin{equation*}
	\mathcal{M}_{J}=\underbrace{\iota^{p^n}\times\dots\times\iota^{p^n}}_{x_n}\times \underbrace{\iota^{p^{n-1}}\times\dots\times\iota^{p^{n-1}}}_{x_{n-1}}\times\dots\times \underbrace{\iota^p\times\dots\times\iota^p}_{x_1},
	\end{equation*} 
	and $\mathcal{M}_{J}$ is the greatest among all monomials which are contained in $P^{\gamma^{-1}\left(J\right)}(\underbrace{\iota\times\dots\times \iota}_{\ex(r)/2})$ with non-zero coefficient. Now note that if $J,J'\in \varUpsilon_r$ and $J>J'$, then  $\mathcal{M}_{J}>\mathcal{M}_{J'}$. Thus, if $J$ is the greatest in $\varUpsilon_r$, then for any $J'\in \varUpsilon_r$ the class $P^{\gamma^{-1}\left(J'\right)}(\underbrace{\iota\times\dots\times \iota}_{\ex(r)/2})$ written as sum of monomials in $\iota$ does not contain the monomial $\mathcal{M}_{J}$. Since the greatest sequence in $\varUpsilon_r$ is  $\gamma(I^r_{\max})$ and $\varepsilon_{\gamma(I^r_{\max})}\ne 0$ by Corollary \ref{tech 100}, then we obtain that the class $\chi(P^r)(\underbrace{\iota\times\dots\times \iota}_{\ex(r)/2})$ written as a sum of monomials in $\iota$ contains with non-zero coefficient the monomial $\mathcal{M}_{\gamma(I^r_{\max})}$ and thus is non-zero. 
\end{proof}	
	
\begin{Lem}\label{number}
	Let $n$ be a positive integer. Then the greatest $k$ such that $n>\ex(k)+2k(p-1)$ is equal to 
	\[
	\sum_{i=1}^{\infty} \left[\frac{n-1}{2p^i}\right].
	\]
\end{Lem}
\begin{proof}
Statement $2$ of Theorem~\ref{combinatorial} implies that if $n>\ex(k)+2k(p-1)$, then $n>\ex(k-1)+2(k-1)(p-1)$. Thus $n>\ex(r)+2r(p-1)$ for all $r\le k$. So the greatest $k$ satisfying $n>\ex(k)+2k(p-1)$ is equal to the number of positive integers $r$ such that $n-1\ge \ex(r)+2r(p-1)$.   

Let $\gamma(I^r_{\max})=(x_1,x_2, \dots, x_m, 0, 0, \dots)$. Since $\gamma(I^r_{\max})\in \varUpsilon_r$, then
\begin{equation}\label{3}
	r=x_1+x_2(1+p)+x_3(1+p+p^2)+\dots+x_m(1+p+\dots+p^{m-1}).
\end{equation} 
Lemma \ref{tech for combinatorial trm} implies that 
\begin{enumerate}[label=\alph*)]
\item $x_i\le p$ for all $i$,
\item  if $x_i=p$, then $x_j=0$ for all $j<i$. 
\end{enumerate}
Moreover, Lemma \ref{tech for combinatorial trm} implies that the representation of $r$ in the form \eqref{3} with conditions a) and b) is unique.
	Besides, from statement $1$ of Theorem~\ref{combinatorial} it follows that $\ex(r)=2\bigl(\sum_{i}x_i\bigr)$. Therefore, we get that 
	$$\ex(r)+ 2r(p-1)=2(px_1+p^2x_2+\dots+p^mx_m).$$
	So we can rewrite inequality $n-1\ge \ex(r)+2r(p-1)$ as 
	\begin{equation} \label{2}
	\frac{n-1}{2p} \ge x_1+px_2+\dots +p^{m-1}x_m.
	\end{equation} 	
	
	Thus,  the greatest $k$ such that $n>\ex(k)+2k(p-1)$ is equal to the number of non-zero sequences~$\{x_i\}$ which satisfy the conditions a), b) and inequality~$\eqref{2}$. Let us refer to the conditions a), b) and inequality~$\eqref{2}$ altogether as to condition $(*)$. 
	
	Denote by $\lambda_0$ the number of non-zero sequences $\{x_i\}$ such that $\{x_i\}$ satisfies condition $(*)$ and  $x_i<p$ for all $i$. Denote by $\lambda_j$ the number of sequences $\{x_i\}$ such that $\{x_i\}$ satisfies condition $(*)$ and  $x_j=p$.
	
From the uniqueness of the base-$p$ positional notation of an integer, it follows that $\lambda_{0}=\bigl[\frac{n-1}{2p}\bigr]$. 

Now, let us compute $\lambda_j$ for $j>0$. Since $x_j=p$, condition b) implies that $x_k=0$ for all $k<j$ and $x_l<p$ for all $l>j$. So inequality \eqref{2} can be rewritten in the form 
	 $$\frac{n-1}{2p^{j+1}} \ge 1+ x_{j+1}+px_{j+2}+\dots +p^{m-1-j}x_m,$$
where $x_{j+s}<p$ for all $s$. Therefore, $\lambda_j=\left[\frac{n-1}{2p^{j+1}}\right]$. 
	 
Thus the number of non-zero sequences $\{x_i\}$ satisfying condition $(*)$ is equal to 
\[
\sum_{j=0}^{\infty} \lambda_j=\sum_{i=1}^{\infty} \left[\frac{n-1}{2p^i}\right]. 
\tag*{\qedhere} 
\]
\end{proof}

\section{Proof of the lower bound}
Up to a homotopy equivalence we can realize the spectrum $\bu_{(p)}(s)$ as a CW $\Omega$-spectrum. We will denote its $q^{\text{th}}$ term by $\bu_{(p)}(q,s)$. In what follows, we assume that $q>1$. So we have that 	
\[
\pi_n\bigl(\bu_{(p)}(q,s)\bigr)=\begin{cases}
\Z_{(p)}, & n=q+2i \ \  \text{with} \ \  0\le i < s; \\
0, & \text{otherwise}.
\end{cases}
\]  	

Since $\bu_{(p)}(q,s)$ is $(q-1)$-connected, we have the canonical isomorphism between  $H^{0}\bigl(\bu_{(p)}(s);\Z_{(p)}\bigr)$ and  $H^{q}\bigl(\bu_{(p)}(q,s);\Z_{(p)}\bigr)$. Denote by $u_{q,s}\in H^q\bigl(\bu_p(q,s);\Z_{(p)}\bigr)$  the image of $u_s \in  H^{0}\bigl(\bu_{(p)}(s);\Z_{(p)}\bigr)$ under this isomorphism. 

\begin{Lem}\label{cohomology of bu p(q,s)}
	There are cohomology classes $a_{s,i}\in H^{q+2i}\bigl(\bu_{(p)}(q,s);\Z_{(p)}\bigr)$ with $1 \le i < s$ such that 
	\begin{align*}
	H^*\bigl(\bu_{(p)}(q,s);\mathbb{Q}\bigr)&\cong\mathbb{Q}[u_{q,s},a_{s,1},\dots,a_{s,s-1}] \quad \  \ \text{if} \ q \  \text{is even},
	\\
	H^*\bigl(\bu_{(p)}(q,s);\mathbb{Q}\bigr)&\cong\Lambda_{\mathbb{Q}}(u_{q,s},a_{s,1},\dots,a_{s,s-1}) \quad \text{if} \ q \  \text{is odd}.
	\end{align*}
\end{Lem}
\begin{proof}
 Denote by $\bu_{(0)}(s)$ the rationalization of  the spectrum $\bu_{(p)}(s)$, and by $\bu_{(0)}(q,s)$ the rationalization of the space $\bu_{(p)}(q,s)$. Since the Postnikov invariants of spectra are of finite order, then $\bu_{(0)}(s)$ is homotopy equivalent to the CW $\Omega$-spectrum $\bigvee_{i=0}^{s-1} \Sigma^{2i}K(\Q)$ with the $q^{\text{th}}$ term $\prod_{i=0}^{s-1} K(\Q, q+2i)$. On the other hand, $\bu_{(0)}(s)$ is  homotopy equivalent to a CW $\Omega$ spectrum with the $q^{\text{th}}$ term homotopy equivalent to $\bu_{(0)}(q,s)$. Thus $\bu_{(0)}(q,s)$ is homotopy equivalent to $\prod_{i=0}^{s-1} K(\Q, q+2i)$. Since 
 \[
 H^*\bigl(\bu_{(p)}(q,s);\mathbb{Q}\bigr)\cong H^*\bigl(\bu_{(0)}(q,s);\mathbb{Q}\bigr),
 \]
we obtain that 
 \begin{align*}
	H^*\bigl(\bu_{(p)}(q,s);\mathbb{Q}\bigr)&\cong\mathbb{Q}[u_{q,s},a'_{s,1},\dots,a'_{s,s-1}] \quad  \ \ \text{if} \ q \  \text{is even},
	\\
	H^*\bigl(\bu_{(p)}(q,s);\mathbb{Q}\bigr)&\cong\Lambda_{\mathbb{Q}}(u_{q,s},a'_{s,1},\dots,a'_{s,s-1}) \quad \text{if} \ q \  \text{is odd},
	\end{align*}
	where $a'_{s,i}\in H^{q+2i}\bigl(\bu_{(p)}(q,s);\Q\bigr)$ for $1\le i < s$.
 Since 
 \[
 H^*\bigl(\bu_{(p)}(q,s);\mathbb{Q}\bigr)\cong H^*\bigl(\bu_{(p)}(q,s);\Z_{(p)}\bigr)\otimes \mathbb{Q},
 \]
  we can multiply $a'_{s,i}$ by integers, so that the resulting classes $a_{s,i}$ will be defined over~$\Z_{(p)}$. 
\end{proof}
\begin{Lem}\label{action of chi(p^r)}
If $q>\ex(r)$, then  $\tilde{\beta}\chi(P^r)$ acts non-trivially on $u_{q,r(p-1)}$.
\end{Lem}
\begin{proof}
Since the Postnikov invariants $k_q^{2m+1}\in H^{q+2m+1}\bigl(\bu_{(p)}(q,m);\Z_{(p)}\bigr)$ of $\bu_{(p)}(q,s)$ have finite order, and the integral cohomology groups of  $K(\Z,2)$ are torsion free, then the obstruction theory yields the existence of a map 
 $$
 f\colon \underbrace{K(\Z,2)\times\dots\times K(\Z,2)}_{\ex(r)/2} \to \bu_{(p)}\bigl(\ex(r),r(p-1)\bigr)
 $$
such that 
\[
f^*(u_{\ex(r),r(p-1)})=\underbrace{\iota\times\dots\times \iota}_{\ex(r)/2},
\]
where $\iota \in H^2(K(\Z,2);\Z_{(p)})$ is the standard generator. Denote the class $\underbrace{\iota\times\dots\times \iota}_{\ex(r)/2}$ by $w$. Let $i$ be a positive integer and $q=\ex(r)+i$.  Denote by $g_i$ the map
\[
 g_i\colon \Sigma^i \bigl(K(\Z,2)^{\ex(r)/2} \bigr) \xrightarrow{\Sigma^i f} \Sigma^i \bu_{(p)}\bigl(\ex(r),r(p-1)\bigr) \to \bu_{(p)}\bigl(q,r(p-1)\bigr).
 \]
 Then $g_i^*(u_{q,r(p-1)})=\Sigma^i w$. Now, let us prove that $\tilde{\beta}\chi(P^r)(u_{q,r(p-1)})$ is nonzero. We will argue by contradiction. If $\tilde{\beta}\chi(P^r)(u_{q,r(p-1)})=0$, then there exists a class 
 $$
 z\in H^{q+2r(p-1)}\Bigl(\bu_{(p)}\bigl(q,r(p-1)\bigl);\mathbb{Q}\Bigr)
 $$
such that $\psi_*z=\rho_*\chi(P^r)(u_{q,r(p-1)})$, where $\psi_*$ is the map on (co)homology groups induced by the quotient homomorphism $\psi\colon \Q \to \Q/\Z_{(p)}$ of the coefficient groups. Since $g_i^*(u_{q,r(p-1)})=\Sigma^i w$, then 
\[
\psi_* g_i^*(z)=\Sigma^i\rho_*\chi(P^r)(w).
\]
By Lemma  \ref{tech for chi(P^r)}, we have that $\chi(P^r)(w)\ne 0$. Since the cohomology groups  $H^*(K(\Z,2)^{\ex(r)/2};\Z)$ are free abelian, then $\rho_*$ is injective and thus $\psi_* g_i^*(z) \ne0$. 
 
On the other hand, by Lemma \ref{cohomology of bu p(q,s)} the cohomology ring $H^{*}\bigl(\bu_{(p)}(q,r(p-1));\mathbb{Q}\bigr)$ is multiplicatively generated by elements of degrees strictly less than $q+2r(p-1)$, but the class $z$ has degree $q+2r(p-1)$. So $z$ can be represented as a sum of decomposable elements in $H^*\bigl(\bu_{(p)}(q,r(p-1));\mathbb{Q}\bigr)$. However, since $i>0$, all the products in the cohomology ring $H^*\bigl(\Sigma^i\bigl(K(\Z,2)^{\ex(r)/2}\bigr);\Q\bigr)$ vanish. Thus, $g_i^*(z)=0$. Contradiction. 
\end{proof}

Consider a map
\[
\pi\colon \bu_{(p)}\bigl(q, s\bigr) \to \prod_{i=0}^{s-1} K(\Q, q+2i)
\]
such that $\pi^*(\iota_{q})=u_{q,s}$ and $\pi^*(\iota_{q+2i})=a_{s,i}$ for $1\le i<s$, where $\iota_n\in H^n\bigl(K(\Q),n);\Q\bigr)$ is the canonical generator. Denote by $R_p(q,s)$ the homotopy fibre of $\pi$. Let 
\begin{equation}\label{9995}
f\colon R_p(q,s) \to  \bu_{(p)}\bigl(q, s\bigr)
\end{equation}
be the map of the fibre and let $\tilde{u}_{q,s}$ denote the class $f^*(u_{q,s})$.

\begin{Lem}\label{538}
The following properties of $R_p(q,s)$ hold.
\begin{enumerate}
\item $\widetilde{H}_*\bigl(R_p(q,s);\Q\bigr)=0$.
\item The operation $\Phi_r$ is defined on $\tilde{u}_{q,r(p-1)}$.
\item If $q>\ex(r)$, then $\tilde{\beta}\chi(P^r)(\tilde{u}_{q,r(p-1)})\ne 0$.
\end{enumerate}
\end{Lem}
\begin{proof}
Lemma \ref{cohomology of bu p(q,s)} and the definition of $\pi$ yields  that $\pi$ induces an isomorphism of rational cohomology groups. Thus $\widetilde{H}_*\bigl(R_p(q,s);\Q\bigr)=0$.

For a map $f$ in Definition \ref{def. def of integer higher operations}, one can take the map \eqref{9995}. So $\Phi_r$ is defined on $\tilde{u}_{q,r(p-1)}$.

Finally, let us prove the third statement. It is well known that 
\[
\widetilde{H}^*\bigl(K(\Q,n);\Z/p^m\Z\bigr)=0
\]
for any positive integers $n$ and $m$. So for any $m$ we have that 
$$
\widetilde{H}^*\Biggl(\prod_{i=0}^{r(p-1)-1} K(\Q, q+2i);\Z/p^m\Z\Biggr)=0.
$$
Therefore, the Serre spectral sequence for the fibration  
\[
\prod\limits_{i=0}^{r(p-1)-1} K(\Q, q+2i-1) \longrightarrow R_{p}\bigl(q,r(p-1)\bigr) \xlongrightarrow{f} \bu_{(p)}\bigl(q,r(p-1)\bigr)
\]
implies that for any $m$ the map
$$
f^*\colon H^*\Bigl(\bu_{(p)}\bigr(q, r(p-1)\bigl);\Z/p^m\Z\Bigr) \to H^*\Bigl(R_p\bigl(q,r(p-1)\bigr);\Z/p^m\Z\Bigr)
$$
is an isomorphism. 

Since $q>\ex(r)$, then Lemma \ref{action of chi(p^r)} implies that $\tilde{\beta}\chi(P^r)(u_{q,r(p-1)})\ne 0$. Note that $\tilde{\beta}\chi(P^r)(u_{q,r(p-1)})$ has order $p$, and the cohomology groups $H^*(\bu_{(p)}\bigl(q,r(p-1)\bigr);\Z_{(p)})$ are finitely generated $\Z_{(p)}$-modules. So there exists $m>0$ such that the reduction of $\tilde{\beta}\chi(P^r)(u_{q,r(p-1)})$ modulo $p^m$ is nonzero. Consider the commutative diagram 
\[
\xymatrix{
H^{q+2r(p-1)+1}\Bigl(\bu_{(p)}\bigl(q,r(p-1)\bigr); \Z_{(p)}\Bigr) \ar[rr]^-{\bmod p^m} \ar[d]^{f^*} && H^{q+2r(p-1)+1}\Bigl(\bu_{(p)}\bigl(q,r(p-1)\bigr); \Z/p^m\Z\Bigr) \ar[d]^{f^*}  \\
H^{q+2r(p-1)+1}\Bigl(R_p\bigl(q,r(p-1)\bigr); \Z_{(p)}\Bigr) \ar[rr]^-{\bmod p^m} && H^{q+2r(p-1)+1}\Bigl(R_p\bigl(q,r(p-1)\bigr); \Z/p^m\Z\Bigr).
}
\]
Since the right vertical map is an isomorphism and $\tilde{\beta}\chi(P^r)(u_{q,r(p-1)})$ modulo $p^m$ is nonzero, then 
\[
0\ne f^*\bigl(\tilde{\beta}\chi(P^r)(u_{q,r(p-1)})\bigr)=\tilde{\beta}\chi(P^r)(\tilde{u}_{q,r(p-1)}).
\tag*{\qedhere}
\]
\end{proof}
Denote by $R(n)$ the space 
$$
 \bigvee_{p \  \text{prime}}R_p\Biggr(n-2\Bigr(\sum_{i=1}^{\infty} \left[\frac{n-1}{2p^i}\right]\Bigl)(p-1),\Bigl(\sum_{i=1}^{\infty} \left[\frac{n-1}{2p^i}\right]\Bigr)(p-1)\Biggl),
$$
where for $q\le0$ we put $R_p(q,s)=\pt$.

\begin{Th}\label{win}
There is a homology class $x\in H_n\bigl(R(n);\Z\bigr)$ such that the minimal positive integer $A$ for which $Ax$ is $\U$-realizable is equal to 
$$
\left[\frac{n-1}{2}\right]\hskip-1mm \mathord{\vcenter{\hbox{\LARGE\textup !}}} \ = \prod_{p\ \text{prime}}p^{\sum_{i=1}^{\infty} \left[\frac{n-1}{2p^i}\right]}.
$$
\end{Th}
\begin{proof}
Denote by $k_p(n)$ the number $\sum_{i=1}^{\infty} \left[\frac{n-1}{2p^i}\right]$. Lemma \ref{number} implies the inequality 
\[
n-2k_p(n)(p-1)>\ex\bigl(k_p(n)\bigr).
\]
So by Lemma \ref{538}, the space $R_p\bigl(n-2k_p(n)(p-1),k_p(n)(p-1)\bigr)$ satisfies the conditions of Lemma~\ref{Lem for example}. Then by Lemma~\ref{Lem for example} there is a class 
$$
	v_p\in H_n\Bigl(R_p\bigl(n-2k_p(n)(p-1),k_p(n)(p-1)\bigr);\Z\Bigr)
$$
such that $mp^{k_p(n)-1}v_p$ is not $\U$-realizable for any $m$ coprime with $p$. Put 
 $$
	 v=\sum_{p \  \text{prime}} v_p\in H_n\bigl(R(n);\Z\bigr).
 $$ 
 Then the minimal $A$ such that the class $Av$ is $\U$-realizable is divisible by $\prod_{p} p^{k_p(n)}$. Multiplying $v$ by a certain integer, we obtain the desired homology class. 
\end{proof}

\section{Proof of the upper bound and computations in low dimensions}

In this section we prove the upper bound from Theorem~\ref{Main theorem intro} and the following theorem.
\begin{Th}\label{exact boundary for small dimensinal classes}	
	If $n<2p^2+2p$, then 
	\[
	\nu_p\bigl(k_{\U}(n)\bigr)=\sum_{i=1}^{\infty}  \left[\frac{n-1}{2p^i}\right].
	\]
\end{Th}
Clearly, Theorem~\ref{exact boundary for small dimensinal classes} yields the following corollary.
\begin{Cor}\label{low dim}
If $n<12$, then
	$$k_{\U}(n)=\prod_{p\ \text{prime}} p^{ \sum_{i=1}^{\infty} \left[\frac{n-1}{2p^i}\right]}=\left[\frac{n-1}{2}\right]\hskip-1mm \mathord{\vcenter{\hbox{\LARGE\textup !}}} .$$
	If $n<24$, then
	$$k_{\SO}(n)=\prod_{p\ \text{prime},\ p>2} p^{ \sum_{i=1}^{\infty} \left[\frac{n-1}{2p^i}\right]}.$$
\end{Cor}
\begin{Rem}\label{integral Steenrod}
Under integral Steenrod operations we mean stable natural transformations of the functor of (co)homology groups with coefficients in $\Z_{(p)}$. By the Yoneda lemma, integral Steenrod operations are in bijection with elements of $H^*\bigl(K(\Z_{(p)}); \Z_{(p)}\bigr)$. Theorem $10.4$ of \cite{M} implies that any integral Steenrod operation $\theta\colon H^*(-;\Z_{(p)}) \to H^{*+s}(-;\Z_{(p)})$ of degree $s>0$ can be written as a composition 
\[
H^*(-;\Z_{(p)}) \xrightarrow{\bmod p} H^*(-;\Z/p\Z) \xrightarrow{\theta_{p}}  H^{*+s-1}(-;\Z/p\Z)  \xrightarrow{\tilde{\beta}} H^{*+s}(-;\Z_{(p)})
\]
where $\theta_{p}\in \mathcal{A}_p$. By analogy, any integral Steenrod operation $\theta\colon H_*(-;\Z_{(p)}) \to H_{*-s}(-;\Z_{(p)})$ of degree~$s>0$ can be written as a composition 
\[
H_*(-;\Z_{(p)}) \xrightarrow{\bmod p} H_*(-;\Z/p\Z) \xrightarrow{\theta_{p}}  H_{*-s+1}(-;\Z/p\Z)  \xrightarrow{\tilde{\beta}} H_{*-s}(-;\Z_{(p)}),
\]
where $\theta_{p}\in \mathcal{A}_p$.
\end{Rem}
 
\begin{Lem} \label{differentials for cobordisms}
In the AHSS for the homology theory $\MU_*\otimes\Z_{(p)}$ of a finite CW complex $X$, for any prime~$p$ and nonnegative integers $t$ and $s$ the following statements hold. 
\begin{enumerate}
		\item The differential $d^t_{*,*}$ vanishes unless $t \equiv 1 \bmod 2(p-1)$. 
		\item If $t=2r(p-1)+1$, then 
\[
		d_{n,0}^{2r(p-1)+1}(p^{r-1}x)=\sum_{i} \theta_{r,i}(x)b_{r(p-1),i} \ \ \ (\text{modulo images of the previous differentials}).
\]
\end{enumerate}
Here, $b_{s,i}\in \Omega^{\U}_{2s}\otimes \Z_{(p)}$ are elements such that the set $\{b_{s,i}\}$ is a $\Z_{(p)}$-basis of $\Omega^{\U}_*\otimes \Z_{(p)}$, 
	$x$ is an element of $E_{n,0}^{2}=H_n(X;\Z_{(p)})$, and 
	$\theta_{r,i}$ are integral Steenrod operations of degree $2r(p-1)+1$.
	\end{Lem}
\begin{proof}
Since the AHSS is stable, we can replace $X$ by its suspension spectrum. Since the Spanier--White-head duality maps the AHSS for $\MU_*\otimes\Z_{(p)}$ of $X$ to the AHSS for $\MU^*\otimes\Z_{(p)}$ of $D(X)$, it is sufficient to prove a cohomological version of the theorem for the spectrum $D(X)$. The cohomological version of statement $1$ of the theorem was proved in Lemma $2$ of  \cite{Buch.1}. 

Now, let us prove statement $2$. Consider a map
\[
f\colon D(X)\to \Sigma^{-n}K\bigl(\Z_{(p)}\bigr)
\]
that classifies the class $D(x)$. Denote by $\iota\in H^0\bigl(K(\Z_{(p)});\Z_{(p)}\bigr)$ the fundamental class. Consider the AHSS for $\MU^*\otimes\Z_{(p)}$ of $K\bigl(\Z_{(p)}\bigr)$. Note that by Remark \ref{integral Steenrod} all nonzero classes in $H^{>0}\bigl(K(\Z_{(p)});\Z_{(p)}\bigr)$ have order~$p$. Also by statement $1$ of the theorem the differential $d_t$ is zero unless $t \equiv 1 \bmod 2(p-1)$. Therefore $p^{r-1}\iota$ lies in $E_{2r(p-1)+1}^{0,0}$. Since cohomology classes in $H^{>0}\bigl(K(\Z_{(p)});\Z_{(p)}\bigr)$ are integral Steenrod operations, then
\[
    d^{0,0}_{2r(p-1)+1}(p^{r-1}\iota)=\sum_{i} \theta_{r,i}(\iota)b^{r(p-1),i}  \ \ \  (\text{modulo images of the previous differentials}),
\]
where $b^{s,i}$ is Spanier--Whitehead dual to $b_{s,i}$. So we get a cohomological version of statement $2$ for the spectrum $K\bigl(\Z_{(p)}\bigr)$.  Finally, the functoriality of AHSS for the map $f$ implies statement $2$ for $D(X)$. 
\end{proof}

\begin{Lem}\label{t6}
If $X$ is a finite CW complex, then the differential $d^{n-1}_{n,0}$ in the AHSS for the theory $\MU_*\otimes \Z_{(p)}$ of $X$ vanishes. 
\end{Lem}
\begin{proof} 
In what follows, the differentials of the AHSS are considered as partially defined multivalued maps on $E_2^{*,*}$--term. We will prove the stronger statement: if $d^{n-1}_{n,0}$ is defined on $x$, then $d^{n-1}_{n,0}(x)=\{0\}$ as sets. We argue by contradiction. Assume that there exists  $x\in H_n(X;\Z_{(p)})$ such that $x\in E^{n-1}_{n,0}$  and $d^{n-1}_{n,0}(x)\ne \{0\}$. Let $z\in H_1\bigl(X;\Omega^{\U}_{n-2}\otimes \Z_{(p)}\bigr)$ be a non-zero class such that  $z\in d^{n-1}_{n,0}(x)$. 

Since the abelian group $\Omega^{\U}_{n-2}\otimes\bigl(\Q/\Z_{(p)}\bigr)$ is divisible, then the universal coefficient theorem yields that the natural map 
$$
    H^1\Bigl(X; \Omega^{\U}_{n-2}\otimes \bigl(\Q/\Z_{(p)}\bigr)\Bigr) \to \Hom\Bigl(H_1\bigl(X; \Z_{(p)}\bigr); \Omega^{\U}_{n-2}\otimes \bigl(\Q/\Z_{(p)}\bigr)\Bigr)
$$
is an isomorphism. Since the group $\Omega^{\U}_{n-2}$ is free abelian and finitely generated, then we have the isomorphism
$$
   \Hom\Bigl(H_1\bigl(X; \Z_{(p)}\bigr); \Omega^{\U}_{n-2}\otimes\bigl(\Q/\Z_{(p)}\bigr)\Bigr) \cong \Hom\Bigl(H_1\bigl(X; \Omega^{\U}_{n-2}\otimes\Z_{(p)}\bigr);  \Q/\Z_{(p)}\Bigr).
$$
Therefore, we get the isomorphism 
$$
   \Psi\colon H^1\Bigl(X; \Omega^{\U}_{n-2}\otimes\bigl(\Q/\Z_{(p)}\bigr)\Bigr) \to \Hom\Bigl(H_1\bigl(X; \Omega^{\U}_{n-2}\otimes\Z_{(p)}\bigr);  \Q/\Z_{(p)}\Bigr),
$$
which is functorial in $X$. Since the group $H_1\bigl(X; \Omega^{\U}_{n-2}\otimes \Z_{(p)}\bigr)$ is a finitely generated $\Z_{(p)}$-module, then there exists $y\in  H^1\Bigl(X; \Omega^{\U}_{n-2}\otimes \bigl(\Q/\Z_{(p)}\bigr)\Bigr)$ such that $\Psi(y)(z)\ne0$.   

Denote by $f\colon X\to K\Bigl(\Omega^{\U}_{n-2}\otimes \bigl(\Q/\Z_{(p)}\bigr),1\Bigr)$ the classifying map for the class $y$. Note that 
 $$
  K\Bigl(\Omega^{\U}_{n-2}\otimes \bigl(\Q/\Z_{(p)}\bigr),1\Bigr) \simeq
  \prod_{\rk\bigl(\Omega^{\U}_{n-2}\bigr)} K\bigl(\mathbb{Q}/\Z_{(p)},1\bigr) \simeq  
  \prod_{\rk\bigl(\Omega^{\U}_{n-2}\bigr)} \colim_{n} K\bigl(\Z/p^n\Z, 1\bigr).
$$
Since the homology functor commutes with direct limits and  $H_{2k}\Bigl(K\bigl(\Z/p^n\Z, 1\bigl);\Z_{(p)}\Bigr)=0$ for any $k\ne0$ and $n>0$ (see \cite{Hatch}, example 2.43), we get that 
\[
H_{2k}\Bigl(K\bigl(\Omega^{\U}_{n-2}\otimes \bigl(\Q/\Z_{(p)}\bigr),1\bigr);\Z_{(p)}\Bigr)=0
\]
for all $k\ne0$. On the other hand, we have that $\Omega^{\U}_{l}=0$ for any odd $l$. Thus all the differentials in the AHSS for $\MU_*\otimes \Z_{(p)}$  of $K\Bigl(\Omega^{\U}_{n-2}\otimes\bigl(\Q/\Z_{(p)}\bigr),1\Bigr)$  vanish. In particular, $d^{n-1}_{n,0}(f_*x)= \{0\}$. Since $$f_*\bigl(d^{n-1}_{n,0}(x)\bigr)\subset d^{n-1}_{n,0}(f_*x),$$ then we get that $f_*\bigl(d^{n-1}_{n,0}(x)\bigr)=\{0\}$. Therefore 
$$
    0\ne \Psi(y)(z)=\Psi(f^*\iota)(z)=\Psi(\iota)(f_*z)=\Psi(\iota)(0)=0,
    $$
   where
$
   \iota \in H^1\Bigl(K\Bigl(\Omega^{\U}_{n-2}\otimes\bigl(\mathbb{Q}/\Z_{(p)}\bigr),1\Bigr);\Omega^{\U}_{n-2}\otimes\bigl(\mathbb{Q}/\Z_{(p)}\bigr)\Bigr) 
$ is the fundamental class. Contradiction. 
\end{proof}   
    
\begin{Th}\label{upper boundry for realization}
Let $X$ be a topological space and $x\in H_n(X;\Z)$. Then the homology class 
\[
\prod_{p\ \text{prime}}p^{\bigl[\frac{n-3}{2p-2}\bigr]}x
\]
is $\U$-realizable. 
\end{Th}
\begin{proof}
From the definition of singular homology it follows that the class $x$ can be realized as an image of a homology class of a finite simplicial 
complex. So we can assume that $X$ is a finite CW complex. Then it is sufficient to prove that for
 any prime p the class $p^{\bigl[\frac{n-3}{2p-2}\bigr]}x$ lies in the $E^{\infty}_{n,0}$--term  of the AHSS for 
 $\MU_*\otimes \Z_{(p)}$ of $X$. Fix a prime $p$ and denote the number $\bigl[\frac{n-3}{2p-2}\bigr]$ by $l$.
  Statement $2$ of  Lemma \ref{differentials for cobordisms} implies that $p^{l}x$ lies in $E^{2(l+1)(p-1)+1}_{n,0}$. 
  Since $2(l+1)(p-1)+1>n-2$, then only differentials $d_{n,0}^{n-1}$ and $d_{n,0}^{n}$ can be nontrivial on $p^{l}x$. But the differential  $d_{n,0}^{n-1}$ vanishes by Lemma \ref{t6}, and the differential $d_{n,0}^{n}$ vanishes since it has finite order and the zero column $E_{0,*}^{*}$  is torsion free.  
\end{proof}
\begin{Lem}\label{7895}
Let $X$ be a finite CW complex and $x\in H_n(X;\Z/p\Z)$. If $\tilde{\beta}(x)\ne 0$, then there exists a class of finite order $y\in H^n(X;\Z_{(p)})$ such that $\langle x, y \rangle \ne 0$. 
\end{Lem}
\begin{proof}
Since $X$ is a finite CW complex, then for any $\Z_{(p)}$-module $M$ we have the exact sequence
\begin{equation}\label{6998}
0 \to \Ext\bigl(H^{m-1}(X; \Z_{(p)}), M \bigr) \to H_m(X;M) \to \Hom\bigl(H^m(X;\Z_{(p)}), M \bigr) \to 0 
\end{equation} 
where the map  $H_m(X;M) \to \Hom\bigl(H^m(X;\Z_{(p)}), M \bigr)$ is given by $a \mapsto \langle a, - \rangle$. Consider the following  commutative diagram
\[
\xymatrix{
0\ar[r]& \Ext\bigl(H^{m-1}(X;\Z_{(p)}), \Z/p\Z\bigr) \ar[r] & H_m(X;\Z/p\Z) \ar[r] & \Hom\bigl(H^m(X;\Z_{(p)}), \Z/p\Z\bigr) \ar[r] & 0  \\
0 \ar[r] & \Ext\bigl(H^{m-1}(X;\Z_{(p)}), \Z_{(p)}\bigr) \ar@{->>}[u] \ar[r] & H_m(X;\Z_{(p)}) \ar[r] \ar[u]^{\bmod p} & \Hom\bigl(H^m(X;\Z_{(p)}),\Z_{(p)}\bigr) \ar[u]_{\psi} \ar[r] & 0
}
\]
where the first and second raws are the sequences \eqref{6998} for $M=\Z/p\Z$ and $M=\Z_{(p)}$ respectively and all the vertical maps are induced by the reduction modulo $p$ map $\Z_{(p)}\to \Z/p\Z$. Since the $\Ext^2$ functor in the category of $\Z_{(p)}$-modules vanish, then the left vertical map is surjective. By diagram chasing it can be checked that if $a\in H_m(X;\Z/p\Z)$ and $\tilde{\beta}(a)\ne 0$ (that is $a$ is not in the image of $\bmod \  p$), then the element 
\[
\langle a, - \rangle \in \Hom\bigl(H^m(X;\Z_{(p)}), \Z/p\Z\bigr)
\]
is not in the image of $\psi$. 

Since $X$ is a finite CW complex, then $H^m(X;\Z_{(p)})\cong F\oplus T$, where $F$ is a finitely generated free \linebreak  $\Z_{(p)}$-module and $T$ is a finite abelian $p$-group. So 
\[
\Hom\bigl(H^m(X;\Z_{(p)}),\Z_{(p)}\bigr)\cong \Hom(F, \Z_{(p)}).
\]
Therefore, a homomorphism $f\colon H^m(X;\Z_{(p)}) \to  \Z/p\Z$ lies in the image of $\psi$ if and only if $f(T)=0$. This concludes the proof. 
\end{proof}
\begin{Lem}\label{tech for small n}
Let $X$ be a finite CW complex and $x\in H_n\bigl(X;\Z_{(p)}\bigr)$. Then the following statements hold. 
\begin{enumerate}
\item If $\theta$ is an integral Steenrod operation of degree $2r(p-1)+1$ with $r\le p$ and $n\le 2rp$, then $\theta(x)=0$. 
\item If $\theta$ is an integral Steenrod operation of degree $2(p+2)(p-1)+1$ and $n<2p^2+2p$,  then $\theta(x)=0$. 
\end{enumerate}
\end{Lem}
\begin{proof} 
Firstly, let us prove statement $1$ for $p=2$ and $r=2$. Then $n\le 8$. We argue by contradiction. Write $\theta$ as $\tilde{\beta}\theta_{2}$ with $\theta_2\in \mathcal{A}_2$. Assume that $\theta(x)\ne0$. Then by Lemma \ref{7895} there exists a class of finite order $y\in H^{n-4}(X;\Z_{(2)})$ such that $\langle \theta_2(x), y\rangle\ne 0$. By \cite{Thom}, \S 3.4 we have the equality $\langle \theta_2(x), y\rangle=\langle x, \chi(\theta_2)(y)\rangle$. Since the degree of $\chi(\theta_2)$ is $4$, then the structure of $\mathcal{A}_2$ yields that 
\[
\chi(\theta_2)=\varepsilon \Sq^4 +\varepsilon' \Sq^3\Sq^1,
\]
 where $\varepsilon, \varepsilon' \in \Z/2\Z$. Since $y$ is defined over $\Z_{(2)}$, then $\Sq^1(y)=0$. So we get that $\langle x, \Sq^4(y)\rangle\ne 0$. If~$n<8$, then $\deg(y)<4$ and so  $\Sq^4(y)=0$. If~$n=8$, then $\Sq^4(y)=y^2$ and so $\langle x, \Sq^4(y)\rangle=\langle x, y^2\rangle= 0$ since the classes $x$ and~$y$ are defined over $\Z_{(2)}$ and $y$ is of finite order. Contradiction. 

For $p=2$, $r=1$ and for $p>2$ statement $1$ can be proved analogously. 

Now let us prove statement $2$ for $p=2$. Then $n<12$.  We argue by contradiction. Write~$\theta$ as $\tilde{\beta}\theta_{2}$ with $\theta_2\in \mathcal{A}_2$. Assume that $\theta(x)\ne0$. Then by Lemma \ref{7895} there exists a class $y\in H^{n-8}(X;\Z_{(2)})$ such that $\langle \theta_2(x), y\rangle\ne 0$. By \cite{Thom}, \S 3.4 we have that $\langle \theta_2(x), y\rangle=\langle x, \chi(\theta_2)(y)\rangle$. Since the degree of $\chi(\theta_2)$ is $8$, then the structure of~$\mathcal{A}_2$ yields that 
\[
\chi(\theta_2)=\varepsilon_0\Sq^8 +\varepsilon_1\Sq^7\Sq^1+\varepsilon_2\Sq^6\Sq^2+ \varepsilon_3\Sq^5\Sq^2\Sq^1 
\]
where $\varepsilon_i \in \Z/2\Z$. Since $y$ is defined over $\Z_{(2)}$, then $\Sq^1(y)=0$. So we get that either $\langle x, \Sq^8(y)\rangle \ne 0$ or $\langle x, \Sq^6\Sq^2(y)\rangle \ne 0$. Since $n< 12$, then $\deg(y)< 4$. Therefore $\Sq^8(y)= \Sq^6\Sq^2(y)=0$ since $\ex(\Sq^8)=8$ and~$\ex(\Sq^6\Sq^2)=4$. 

For $p>2$ statement $2$ can be proved analogously. 
\end{proof}
Finally, let us prove Theorem~\ref{exact boundary for small dimensinal classes}.
\begin{proof}[Proof of Theorem \ref{exact boundary for small dimensinal classes}]
	Theorem~\ref{win} implies that 
	$$\nu_p\bigl(k(n)\bigr)\ge\sum_{i=1}^{\infty}  \left[\frac{n-1}{2p^i}\right].$$
	Therefore, it is sufficient to prove that for a homology class $x\in H_n(X;\Z)$ of a finite CW complex $X$, the homology class $p^{\sum_{i=1}^{\infty}  \left[\frac{n-1}{2p^i}\right]}x$ lies in $E^{\infty}_{n,0}$ of the AHSS for $\MU_*\otimes \Z_{(p)}$ of $X$, provided that $n<2p^2+2p$. 
	
Firstly, assume that $2(r-1)p+1\le n \le 2rp$ with $1\le r\le p$. Then $\sum_{i=1}^{\infty}  \left[\frac{n-1}{2p^i}\right]=r-1$ and Lemma \ref{differentials for cobordisms} implies that  
		\[
		p^{r-1} x \in E^{2r(p-1)+1}_{n,0}
		\]
		and 
		\begin{equation}\label{dif1} 
		d_{n,0}^{2r(p-1)+1}(p^{r-1} x)=\sum_{j}\theta_{r,j}(x)b_{r(p-1),j}
		\ \  (\text{modulo images of the previous differentials}).
		\end{equation}	
Since $\dim(x)=n\le 2rp$ with $r\le p$ and the degree of the integral Steenrod operations $\theta_{r,j}$ is equal to $2r(p-1)+1$, then statement $1$ of Lemma \ref{tech for small n} implies that the differential \eqref{dif1} is zero. The next possibly nontrivial differential has degree $2(r+1)(p-1)+1\ge n+2(p-r)-1$ with $r\le p$. If $r<p$, then this and all the next differentials are trivial by dimensional reasons. If $r=p$ and $2(r+1)(p-1)+1=n-1$, then this differential is trivial by Lemma \ref{t6} and all the next differentials are trivial by dimensional reasons.

Finally, assume that $2p^2+1\le n< 2p^2+2p$. Then $\sum_{i=1}^{\infty}  \left[\frac{n-1}{2p^i}\right]=p+1$ and  Lemma \ref{differentials for cobordisms} implies that 
\[
p^{p+1} x \in E^{2(p+2)(p-1)+1}_{n,0}
\]
and 
\begin{equation}\label{dif2} 
d_{n,0}^{2(p+2)(p-1)+1}\bigl(p^{p+1} x\bigr)=\sum_{j} \theta_{p+2,j}(x)b_{(p+2)(p-1),j}
\ \ (\text{modulo images of the previous differentials}).		
\end{equation}	
Since $\dim(x)=n<2p^2+2p$ and the degree of the integral Steenrod operations $\theta_{p+2,j}$ is equal to \linebreak $2(p+2)(p-1)+1$, then statement $2$ of Lemma \ref{tech for small n} implies that the differential \eqref{dif2} is zero. The next possibly nontrivial differential has degree $2(p+3)(p-1)+1\ge n+2p-4\ge n$. If $p>2$, then this and all the next differentials are trivial by dimensional reasons. If $p=2$ and $2(p+3)(p-1)+1=n$, then this differential is trivial since it has finite order and the  zero column $E_{0,*}^{*}$  is torsion free and all the next differentials are trivial by dimensional reasons.
\end{proof}

 \bibliographystyle{plain}
\bibliography{../MainBib}

\end{document}